\documentclass{gtpart}
\usepackage{latexsym}
\usepackage{rotating}
\usepackage{amsfonts}
\usepackage{amssymb}
\usepackage{amsmath}
\usepackage{graphics, color}
\input xy
\xyoption{all}

\title[Yang--Mills theory and the Atiyah--Segal theorem]{Yang--Mills theory over surfaces and the \\
Atiyah--Segal theorem}
\author{Daniel A. Ramras}
\givenname{Daniel}
\surname{Ramras}
\address{Department of Mathematics\\
Vanderbilt University\\
\newline Nashville, TN 37240 \\U.S.A.}
\email{daniel.a.ramras@vanderbilt.edu}
\urladdr{http://www.math.vanderbilt.edu/~ramrasda}

 
 
 \arxivreference{0710.0681}  
\arxivpassword{mzsif}

\newtheorem{theorem}{Theorem}[section]
\newtheorem{lemma}[theorem]{Lemma}
\newtheorem{corollary}[theorem]{Corollary}
\newtheorem{proposition}[theorem]{Proposition}

\newtheorem{conjecture}[theorem]{Conjecture}

\newtheorem{definition}[theorem]{Definition}
\newtheorem{remark}[theorem]{Remark}

\DeclareMathOperator*{\colim}{colim}
\DeclareMathOperator*{\hocolim}{telescope}

\DeclareMathOperator*{\coprodmo}{\coprod}

\newcommand{\bbA}{\mathbb{A}}
\newcommand{\bbC}{\mathbb{C}}

\newcommand{\bbN}{\mathbb{N}}
\newcommand{\bbR}{\mathbb{R}}
\newcommand{\bbZ}{\mathbb{Z}}

\newcommand{\A}{\mathcal{A}}

\newcommand{\mcC}{\mathcal{C}}
\newcommand{\mcD}{\mathcal{D}}
\newcommand{\E}{\mathcal{E}}

\newcommand{\G}{\mathcal{G}}

\newcommand{\mH}{\mathcal{H}}
\newcommand{\mcH}{\mathcal{H}} 
\newcommand{\I}{\mathcal{I}}

\newcommand{\mcM}{\mathcal{M}}

\newcommand{\bG}{\mathcal{G}_0} 

\newcommand{\leqs}{\leqslant}
\newcommand{\geqs}{\geqslant}
\newcommand{\heq}{\simeq}

\newcommand{\maps}{\longrightarrow}
\newcommand{\injects}{\hookrightarrow}
\newcommand{\homeo}{\cong}
\newcommand{\surjects}{\twoheadrightarrow}
\newcommand{\isom}{\cong}
\newcommand{\cross}{\times}

\newcommand{\wt}[1]{\widetilde{#1}} 
\newcommand{\fc}{\mathcal{A}_{\mathrm{flat}}} 

\newcommand{\vect}[1]{\stackrel{\rightharpoonup}{\mathbf #1}}

\newcommand{\Rep}{\mathrm{Rep}}

\newcommand{\Hom}{\mathrm{Hom}}

\newcommand{\K}{K^{\mathrm{def}}}
\newcommand{\Rdef}{R^{\mathrm{def}}}

\newcommand{\Css}{\mathcal{C}_{ss}}
\newcommand{\Map}{\mathrm{Map}}
\newcommand{\flatc}{\mathcal{A}_{\mathrm{flat}}}

\def\co{\colon\thinspace}

\begin{document} 
 
 \begin{abstract}
In this paper we explain how Morse theory for the Yang--Mills functional can be used to prove an analogue, for surface groups, of the Atiyah--Segal theorem.  Classically, the Atiyah--Segal theorem relates the representation ring $R(\Gamma)$ of a compact Lie group $\Gamma$ to the complex $K$--theory of the classifying space $B\Gamma$.  For infinite discrete groups, it is necessary to take into account deformations of representations, and with this in mind we replace the representation ring by Carlsson's deformation $K$--theory spectrum $\K (\Gamma)$ (the homotopy-theoretical analogue of $R(\Gamma)$).  Our main theorem provides an isomorphism in homotopy 
$\K_*(\pi_1 \Sigma)\isom K^{-*}(\Sigma)$ for all compact, aspherical surfaces $\Sigma$ and all $*>0$.  Combining this result with work of Tyler Lawson, we obtain homotopy theoretical information about the stable moduli space of flat unitary connections over surfaces.
\end{abstract}

\maketitle{}


\section{Introduction}

In this paper we present evidence of two newly emerging phenomena involving the representation spaces $\Hom(\Gamma, U(n))$, for finitely generated discrete groups $\Gamma$ admitting compact classifying spaces.  The first phenomenon is akin to the classical Atiyah--Segal theorem, and relates Carlsson's deformation $K$--theory spectrum $\K(\Gamma)$ (an object built from the representation spaces) to the topological $K$--theory of the classifying space $B\Gamma$.  Second is that in dimensions higher than the cohomological dimension of $\Gamma$, the homotopy groups of the coarse moduli space $\mcM_n (\Gamma) = \Hom(\Gamma, U(n))/U(n)$
appear to vanish after an appropriate stabilization.  In some cases, including the surface groups considered here, this stabilization simply amounts to forming the colimit
$$\mcM (\Gamma) = \colim_{n\to\infty} \mcM_n (\Gamma) = \Hom(\Gamma, U)/U.$$
If $\Gamma$ is the fundamental group of a compact manifold $M$, this space may also be viewed as the stable moduli space of flat connections (or Hermitian bundles) over $M$.  

The link between these phenomena is provided by recent work of Tyler Lawson~\cite{Lawson-prod, Lawson-simul}, which shows that $\mcM (\Gamma)$ is closely related to the cofiber of the Bott map in deformation $K$--theory (see Section~\ref{coarse-moduli}).  Lawson applied his theorems to prove that for finitely generated free groups $F_k$, one has a weak equivalence (of $\mathbf{ku}$--algebras) $\K(F_k)\heq \Map(BF_k, \mathbf{ku})$, and that the stable moduli space $\mcM (F_k) = U^k/U$ is homotopy equivalent to the torus $(S^1)^k$.  Lawson's work also provides analogous statements for free abelian groups (note though that $\mcM (\bbZ^k)$ may be computed by hand, using the fact that commuting matrices are simultaneously diagonalizable).
In this paper we use Morse theory for the Yang--Mills functional to prove:

\medskip
\noindent {\bf Theorem \ref{main-thm}}\quad
{\sl Let $M$ be a compact, aspherical surface.  Then for $*>0$,
$$\K_*(\pi_1 (M))\isom K^{-*}(M).$$
}

Bott periodicity provides an isomorphism $K^{-*} (X) \isom K^*(X)$ for any space $X$, but the isomorphism constructed in this paper more naturally lands in 
$\pi_* \Map(M, \bbZ\cross BU) = K^{-*} (M)$.
This result yields a complete computation of $\K_*(\pi_1 M)$ (Corollary~\ref{k-groups}).  The isomorphism is natural for smooth maps between surfaces, and is in particular equivariant with respect to the mapping class group of the surface.
For non-orientable surfaces, there is actually an isomorphism on $\pi_0$ as well; this is just a re-interpretation of the results of Ho and Liu~\cite{Ho-Liu-ctd-comp-I, Ho-Liu-ctd-comp-II}.  When $M = S^1\cross S^1$, Theorem~\ref{main-thm} follows from T. Lawson's product formula $\K(\Gamma_1\cross \Gamma_2) \heq \K(\Gamma_1)\wedge_{\mathbf{ku}} \K(\Gamma_2)$~\cite{Lawson-prod} together with his calculation of $\K(\bbZ)$ as a $\mathbf{ku}$--module~\cite{Lawson-simul}.  
As we will explain, Theorem~\ref{main-thm} provides evidence that the homotopy groups of $\mcM(\pi_1 \Sigma)$ vanish above dimension 2 in the orientable case and above dimension 1 in the non-orientable case; it also allows us to compute the non-zero groups (Corollary~\ref{moduli} and Proposition~\ref{moduli2}).

Theorem~\ref{main-thm} is closely related to the classical theorem of Atiyah and Segal~\cite{Atiyah-char-coh, Atiyah-Segal}.  For a compact Lie group $G$, 
the topological $K$--theory of the infinite complex $B\Gamma$ is the limit of the $K$--theories of the skeleta 
$B\Gamma^{(n)}$, and hence has the structure of a complete ring (this follows most readily from a 
$\lim^1$ calculation).  The Atiyah--Segal theorem states that this ring is isomorphic to the completion of the representation ring $R(\Gamma)$ at its augmentation ideal (the virtual representations of virtual dimension zero).  

For groups $\Gamma$ with $B\Gamma$ compact, $K^*(B\Gamma)$ is no longer complete and one might hope to relate $K(B\Gamma)$ directly to $R(\Gamma)$.
Deformation $K$--theory may be viewed as the direct homotopical analogue of $R(\Gamma)$ (see Section~\ref{deformations}), so Theorem~\ref{main-thm} should be viewed as a direct homotopical analogue of the Atiyah--Segal theorem.
This suggests that deformation $K$--theory is the proper setting in which to study Atiyah--Segal phenomena for groups with compact classifying spaces, and
we expect that for many such groups $\Gamma$, the deformation $K$-groups of $\Gamma$ will agree with $K^{-*}(B\Gamma)$ for $*$ greater than the cohomological dimension of $\Gamma$ minus one.  The author's excision result for free products~\cite{Ramras-excision} and Lawson's product formula~\cite{Lawson-prod} indicate that this phenomenon should be stable under both free and direct products of discrete groups.  In particular, an Atiyah--Segal theorem for free products of surface groups follows immediately from Theorem~\ref{main-thm} together with the main result from~\cite{Ramras-excision} (or can be deduced from the proof of Theorem~\ref{main-thm}). 
The extent of this relationship is at yet unclear.  In all known examples, deformation $K$--theory eventually becomes 2--periodic, but there are examples in which these periodic groups fail to agree with $K$--theory of the classifying space (see Section~\ref{deformations}).  

Extensions and analogues of the Atiyah--Segal theorem (or more generally, the relation between representations and $K$--theory) have been studied extensively.  For infinite discrete groups $\Gamma$ satisfying appropriate finiteness conditions, Adem~\cite{Adem} and L{\"uck}~\cite{Luck} have studied the relationship between the $K$--theory of the classifying space $B\Gamma$ and the representation rings of the finite subgroups of $\Gamma$.  L{\"u}ck and Oliver~\cite{Luck-Oliver} considered the case of an infinite discrete group $\Gamma$ acting properly, i.e. with finite stabilizers, on a space $X$.  They showed that the $\Gamma$--equivariant $K$--theory of $X$, completed appropriately, agrees with the topological $K$--theory of the homotopy orbit space $E\Gamma \cross_\Gamma X$.  When $X$ is a point, properness forces $\Gamma$ to be finite, and the L{\"u}ck--Oliver theorem reduces to the Atiyah--Segal theorem.  For finite groups $\Gamma$, Chris Dwyer has recently established a twisted version of the classical Atiyah--Segal theorem, relating twisted $K$--theory of $B\Gamma$ to the completion of a twisted version of $R(\Gamma)$.  Deformation $K$--theory should also prove useful in studying Atiyah--Segal phenomena for groups with non-compact classifying spaces, and in this context, Carlsson's derived completion~\cite{Carlsson-derived} should play the role of the completed representation ring.  This approach should lead to a spectrum-level version of the Atiyah--Segal theorem itself, and may also yield spectrum-level versions of these various extensions.  

In a different direction, the Baum-Connes conjecture relates an analytical version of the representation ring of a group $\Gamma$ (namely the reduced $C^*$--algebra of $\Gamma$) to the equivariant $K$--homology of the classifying space for proper actions.
It is interesting to note that for a non-orientable surface $\Sigma$, deformation $K$--theory recovers the topological $K$--theory of $\Sigma$, whereas the $K$--theory of $C^*_{\rm red} (\Gamma)$ is the $K$--homology of $B\Gamma$.  A direct relationship between deformation $K$-theory and the $C^*$--algebras, at least for groups with no torsion in their (co)homology, would be extremely interesting.

The failure of Theorem~\ref{main-thm} in degree zero (and the failure in higher degrees for tori~\cite{Lawson-prod}) is an important feature of deformation $K$--theory, and reflects its close ties to the topology of representation spaces.  While $K$--theory is a stable homotopy invariant of $G$ (i.e. depends only on the stable homotopy type of $BG$), the representation spaces carry a great deal more information about the group $G$, and some of this information is captured by the low-dimensional deformation $K$--groups.  Hence deformation $K$--theory should be viewed as a subtler invariant of $G$, and its relationship to the topological $K$--theory of $BG$ should be viewed as an important computational tool.

As an application of Theorem~\ref{main-thm} (and a justification of the preceding paragraph), we obtain homotopy-theoretical information about the stable moduli space $\mcM(\pi_1 M)$ of flat unitary connections over a compact, aspherical surface $M$.  Ho and Liu~\cite{Ho-Liu-non-orient, Ho-Liu-moduli} computed the components of this moduli space before stabilization (for general structure groups).  In Section~\ref{coarse-moduli}, we combine our work with T. Lawson's results on the Bott map in deformation $K$--theory~\cite{Lawson-simul} to study these moduli spaces after stabilizing with respect to the rank.  In particular, we prove:

\begin{corollary} Let $M$ be a compact, aspherical surface.  Then the fundamental group of the stable moduli space $\mcM (\pi_1 M)$ is isomorphic to $K^{-1} (\Sigma) \isom \K_1 (\pi_1 \Sigma)$, and if $M$ is orientable then $\pi_2 \mcM (\pi_1 M) \isom \bbZ$.
\end{corollary}

Lawson's results naturally lead to a conjectural description of the homotopy type of $\mcM(\pi_1 M)$ (for a compact, aspherical surfaces $M$) .  For an orientable surface $M^g$, we expect that $\mcM(\pi_1 M^g) \heq \mathrm{Sym}^\infty (M^g)$; for a non-orientable surface $\Sigma$ we expect a homotopy equivalence $\mcM(\pi_1 \Sigma) \heq (S^1)^k \coprod (S^1)^k$, where $k$ is the rank $H^1(\Sigma; \bbZ)$.

Theorem~\ref{main-thm} relies on Morse theory for the Yang--Mills functional, as devoped by Atiyah and Bott~\cite{A-B}, Daskalopoulos~\cite{Dask}, and R{\aa}de~\cite{Rade}; the key analytical input comes from Uhlenbeck's compactness theorem~\cite{Uhl, Wehrheim}.  The non-orientable case uses recent work of Ho and Liu~\cite{Ho-Liu-non-orient, Ho-Liu-ctd-comp-II} regarding representation spaces and Yang--Mills theory for non-orientable surfaces.  
Deformation $K$--theory and Yang--Mills theory are connected by the well-known fact that representations of the fundamental group induce flat connections, which form a critical set for the Yang--Mills functional.  In Section~\ref{S^1}, we motivate our arguments by giving a proof along these lines that $\K_*(\bbZ) \isom K^{-*}(S^1)$. 

This paper is organized as follows.  In Section~\ref{K-def}, we introduce and motivate deformation $K$--theory, and explain how the McDuff--Segal group completion theorem provides a convenient model for the zeroth space of the $\Omega$--spectrum $\K (\pi_1 M)$.  In Section~\ref{rep-flat-section}, we discuss the precise relationship between representation varieties and spaces of flat connections.  In Section~\ref{Harder-Narasimhan} we discuss the Harder-Narasimhan stratification on the space of holomorphic structures and its relation to Morse theory for the Yang--Mills functional.  The main theorem is proven in Section~\ref{main-thm-section}, using the results of Sections~\ref{K-def}, \ref{rep-flat-section}, and \ref{Harder-Narasimhan}. 
Section~\ref{coarse-moduli} discusses T. Lawson's work on the Bott map and its implications for the stable moduli space of flat connections.  In Section~\ref{excision}, we explain how the failure of Theorem~\ref{main-thm} in degree zero leads to a failure of excision for connected sum decompositions of Riemann surfaces.  Finally, we have included an appendix discussing the holonomy representation associated to a flat connection.

 {\bf Acknowledgments.}
   
Parts of this work first appeared in my Stanford University Ph.D. thesis~\cite{Ramras-thesis}.  I thank my advisor, Gunnar Carlsson, who suggested this topic and provided crucial encouragement and advice.  I also thank T. Baird, R. Cohen, Z. Fiedorowicz, R. Friedman, I. Leary, T. Lawson, R. Lipshitz, M. Liu, R. Mazzeo, D. Morris, B. Parker, R. Schoen, and K. Wehrheim for helpful conversations, and N.-K. Ho, J. Stasheff, and the referee for comments on previous drafts.  Portions of this article were written at Columbia University, and I would like to thank the Columbia mathematics department for its hospitality.  The appendix was expertly typed by Bob Hough.

This research was partially supported by an NDSEG fellowship, from the Department of Defense, and later by an NSF graduate research fellowship.
    

\section{Deformation K-theory}$\label{K-def}$
    
In this section, we motivate and introduce Carlsson's notion of deformation $K$--theory~\cite{Carlsson-derived-rep} and discuss its basic properties.  Deformation $K$--theory is a contravariant functor from discrete groups to spectra, and is meant to capture homotopy-theoretical information about the representation spaces of the group in question.  As first observed by T. Lawson~\cite{Lawson-thesis}, this spectrum may be constructed as the $K$--theory spectrum associated to a topological permutative category of representations (for details see~\cite[Section 2]{Ramras-excision}). 
Here we will take a more naive, but essentially equivalent, approach.  The present viewpoint makes clear the precise analogy between deformation $K$--theory and the classical representation ring.  

\subsection{Deformations of representations and the Atiyah--Segal theorem}$\label{deformations}$

Associated to any group $\Gamma$, one has the (unitary) representation ring $R(\Gamma)$, which consists of ``virtual isomorphism classes" of representations.  Each representation $\rho\co\Gamma \to U(n)$ induces a vector bundle $E_{\rho} = E\Gamma \cross_{\Gamma} \bbC^n$ (where $\Gamma$ acts on $E\Gamma$ by deck transformations and on $\bbC^n$ via the representation $\rho$) over the classifying space $B\Gamma$, and this provides a map $R(\Gamma) \stackrel{\alpha}{\to} K^0(B\Gamma)$.  When $\Gamma$ is a compact Lie group, the Atiyah--Segal theorem states that $\alpha$ becomes an isomorphism after completing $R(\Gamma)$ at its augmentation ideal.  

Now consider the simplest infinite discrete group, namely $\Gamma = \bbZ$.  Representations of $\bbZ$ are simply unitary matrices, and isomorphism classes of representations are conjugacy classes in $U(n)$.  By the spectral theorem, these conjugacy classes correspond to points in the symmetric product $\mathrm{Sym}^n(S^1)$, and the natural map $\coprod_n \mathrm{Sym}^n(S^1) \to R(\bbZ)$ is injective.  So the discrete representation ring of $\bbZ$ is quite large, and bears little relation to $K$--theory of $B\bbZ = S^1$: every complex vector bundle over $S^1$ is trivial, and so $K^0(S^1)$ is just the integers.  

In this setting, deformations of representations play an important role. A deformation of a representation $\rho_0\co \Gamma\to U(n)$ is simply a representation $\rho_1$ and a continuous path of representations $\rho_t$ connecting $\rho_0$ to $\rho_1$.  The path $\rho_t$ now induces a bundle homotopy $E_{\rho_t}$ between $E_{\rho_0}$ and $E_{\rho_1}$.  Hence the bundle associated to $\rho_0$ is \emph{isomorphic} to the bundle associated to each of its deformations, and the map from representations to $K$--theory factors through deformation classes.\footnote{
For finite groups, this discussion is moot: any deformation of a representation $\rho$ is actually isomorphic to $\rho$, because the trace of a representation gives a continuous, complete invariant of the isomorphism type, and on representations of a fixed dimension, the trace takes on only finitely many values.  Hence when $G$ is finite, deformations are already taken into account by the construction of $R(G)$.
}  
Returning to the example $\Gamma = \bbZ$, we observe that since $U(n)$ is path connected, the natural map from deformation classes of representations to $K$--theory of $B\bbZ$ group completes to an isomorphism. 

With this situation understood, one is inclined to look for an analogue of the representation ring which captures deformations of representations, i.e. the topology of representation spaces.  
The most naive approach fails rather badly: the monoid of deformation classes $\coprod_n \pi_0 \Hom(\Gamma, U(n))$ admits a well-defined map to $K^0(B\Gamma)$, but (despite the case $\Gamma = \bbZ$) this map does not usually group-complete to an isomorphism: the representation spaces $\Hom(\Gamma, U(n))$ are compact CW-complexes, so have finitely many components, but there can be infinitely many isomorphism types of $n$--dimensional bundles over $B\Gamma$.  In the case of Riemann surfaces (i.e. complex curves) $\Sigma$, the spaces $\Hom(\pi_1 \Sigma, U(n))$ are always connected (see discussion at the end of Section~\ref{gp-comp}), so the monoid of deformation classes is just $\bbN$ and its group completion is $\bbZ$; on the other hand bundles over a Riemann surface are determined by their dimension and first Chern class (and all Chern classes are realized) so $K^0(\Sigma) = \bbZ\oplus \bbZ$.  Note here that $\Sigma = B(\pi_1 \Sigma)$ except in the case of the Riemann sphere.

The deformation-theoretical approach is not doomed to failure, though. 
Let $\Rep(\Gamma)$ denote the topological monoid of unitary representation spaces, and let $\mathrm Gr$ denote the Grothendieck group functor.
Carlsson's deformation $K$--theory spectrum $\K (\Gamma)$~\cite{Carlsson-derived-rep} is a lifting of the functor $\mathrm{Gr} \left(\pi_0 \Rep(\Gamma)\right)$ to the category of spectra, or in fact, $\mathbf{ku}$--algebras, in the sense that
$$\pi_0 \K(\Gamma) \isom\mathrm{Gr} \left(\pi_0 \Rep(\Gamma)\right).$$
The isomorphism $\K_*(\pi_1 (M))\isom K^{-*}(M)$ for $*>0$ in Theorem~\ref{main-thm} may be seen as a correction to the fact that $\mathrm{Gr} \left(\pi_0 \Rep(\pi_1 \Sigma)\right) \maps K^0(\Sigma)$
fails to be an isomorphism when $\Sigma$ is a compact, aspherical Riemann surface.  

We conclude this section by noting two cases in which deformation $K$-theory fails to agree with topological $K$-theory, even in high dimensions.  Lawson showed~\cite[Example 34]{Lawson-prod} that the deformation $K$--theory of the integral Heisenberg group $H^3$, whose classifying space is a 3--manifold, is 2--periodic starting in dimension 1.   However, its homotopy groups are infinitely generated.  
As pointed out to me by Ian Leary, there also exist groups $\Gamma$ which have no finite dimensional unitary representations, but \emph{do} have non-trivial $K$--theory.  Leary's examples arise from Higman's group~\cite{Higman} 
$$H = \langle a, b, c, d\,\, | \,\,a^b = a^2, \,\, b^c= b^2, \,\, c^d = c^2,\,\, d^a = d^2 \rangle,$$
which has no finite quotients, and hence no finite-dimensional unitary representations (a representation would give a linear quotient, and by a well-known theorem of Malcev, finitely generated linear groups are residually finite).  Now, Higman's group has the (co)homology of a point,\footnote{This can be proven using the Mayer--Vietoris sequence for a certain amalgamation decomposition of $H$~\cite{BDH}, or from the fact that the presentation 2-complex for $H$ is a model for $BH$ (with 4 one-cells and 4 two-cells).  The latter fact follows from Higman's proof that $H \neq \{1\}$.} so can now build a Kan--Thurston group for $S^2$ by amalgamating two copies of $H$ along the infinite cyclic subgroups generated, say, by $a\in H$.  The resulting group $G$ still admits no unitary representations but now has the integral (co)homology, hence the $K$--theory, of a 2--sphere. In this case $\K(G) = \K(\{1\}) = \mathbf{ku}$.

\subsection{The construction of deformation K--theory}$\label{K-def2}$

For the rest of this section, we fix a discrete group $\Gamma$.  The construction of the (unitary) representation ring $R(\Gamma)$ may be broken down into several steps: one begins with the \emph{sets} $\Hom(\Gamma, U(n))$, which form a monoid under direct sum; next, one takes isomorphism classes by modding out the actions of the groups $U(n)$ on the sets $\Hom(\Gamma, U(n))$.  The monoid structure descends to the quotient, and in fact tensor product now induces the structure of a semi-ring.  Finally, we form the Grothendieck ring $R(\Gamma)$ of this semi-ring of isomorphism classes.  Deformation $K$--theory (additively, at least) may be constructed simply by replacing each step in this construction by its homotopy theoretical analogue.  To be precise, we begin with the \emph{space} 
$$\Rep(\Gamma) = \coprod_{n=0}^{\infty} \Hom(\Gamma, U(n)),$$
which is a topological monoid under block sum.  Rather than passing to $U(n)$--orbit spaces, we now form the homotopy quotient 
$$\Rep(\Gamma)_{hU} = \coprod_{n=0}^{\infty} EU(n) \cross_{U(n)} \Hom(\Gamma, U(n)).$$
Block sum of unitary matrices induces maps $EU(n)\cross EU(m) \to EU(n+m)$, and together with the monoid structure on $\Rep(\Gamma)$ these give $\Rep(\Gamma)_{hU}$ the structure of a topological monoid (for associativity to hold, we must use a functorial model for $EU(n)$, as opposed to the infinite Stiefel manifolds; see Remark~\ref{mixed}).  Finally, we apply the homotopical version of the Grothendieck construction to this topological monoid and call the resulting space $\K(\Gamma)$, the (unitary) deformation $K$--theory of $\Gamma$. 

\begin{definition} The deformation $K$--theory of a discrete group $\Gamma$ is the space
$$\K(\Gamma) := \Omega B \left(\Rep(\Gamma)_{hU}\right),$$
whose homotopy groups we denote by $\K_*(\Gamma) = \pi_* \K(\Gamma)$.
\end{definition}

Here $B$ denotes the simplicial bar construction, namely the classifying space of the topological category with one object and with $\Rep(\Gamma)_{hU}$ as its space of morphisms.  Note that $\K(\Gamma)$ is a contravariant functor from discrete groups to spaces.

It was shown in~\cite[Section 2]{Ramras-excision} that the above space $\K(\Gamma)$ is weakly equivalent to the zeroth space of the connective $\Omega$--spectrum associated to T. Lawson's topological permutative category of unitary representations; in particular the homotopy groups of this spectrum agree with the homotopy groups of the space $\K(\Gamma)$.
We note that constructing a ring structure in deformation $K$--theory requires a subtler approach, and this has been carried out by T. Lawson~\cite{Lawson-prod}.

The first two homotopy groups of $\K(\Gamma)$ have rather direct meanings: $\K_0(\Gamma)$ is the Grothendieck group of virtual connected components of representations, i.e. $\mathrm{Gr} \left(\pi_0 \Rep(\Gamma)\right)$~\cite[Section 2]{Ramras-excision}.  It follows from work of Lawson~\cite{Lawson-simul} that the group $\K_1(\Gamma)$ is a stable version of the group $\pi_1 \Hom(\Gamma, U(n))/U(n)$; a precise discussion will be given in Section~\ref{coarse-moduli}.

\begin{remark}\label{mixed}  
In~\cite{Ramras-excision}, the simplicial model $E^S U(n)$ for $EU(n)$ is used; in this paper we will need to use universal bundles for Sobolev gauge groups, where the simplicial model may not give an actual universal bundle.  Hence it is more convenient to use Milnor's infinite join construction $E^J U(n)$~\cite{Milnor-univ-bundles-2}, which is functorial and applies to all topological groups.  These two constructions are related by the ``mixed model" $B^M U(n) = \left(E^S U(n)\cross E^J U(n) \right)/U(n)$, which maps by weak equivalences to both versions of $BU(n)$.
\end{remark}

\subsection{Group completion in deformation K-theory}$\label{gp-comp}$

The starting point for our work on surface groups is an analysis of the consequences of McDuff--Segal Group Completion theorem~\cite{McDuff-Segal} for deformation $K$--theory, as carried out in~\cite{Ramras-excision}.  Here we recall that result and explain its consequences for surface groups.  Given a topological monoid $M$ and an element $m\in M$, we say that $M$ is \emph{stably group-like} with respect to $m$ if the submonoid of $\pi_0 M$ generated by the component containing $m$ is cofinal (in $\pi_0 M$).  Explicitly, $M$ is stably-group-like with respect to $m$ if for every $x\in M$, there exists an element $x^{-1}\in M$ such that $x\cdot x^{-1}$ is connected by a path to $m^n$ for some $n\in \bbN$.  We then have:

\begin{theorem}[\cite{Ramras-excision}] $\label{gp-completion-cor}$
Let $\Gamma$ be a finitely generated discrete group such that $\Rep(\Gamma)$ is stably group-like with respect to $\rho \in \Hom(\Gamma, U(k))$.
Then there is a weak equivalence  
$$ \K(\Gamma) \heq \hocolim \left( \Rep(\Gamma)_{hU} \stackrel{\oplus \rho}{\maps} \Rep(\Gamma)_{hU} 
			\stackrel{\oplus\rho}{\maps} \cdots \right),
$$
where $\oplus \rho$ denotes block sum with the point $[*_k, \rho]\in EU(k) \cross_{U(k)} \Hom(\Gamma, U(k)$.\end{theorem}

Here, and throughout this article, $\hocolim$ refers to the mapping telescope of a sequence of maps.
The novel aspect of this result is that, unlike elsewhere in algebraic $K$--theory, Quillen's +-construction does not appear.  This is due to the fact that the fundamental group on the right-hand side is already abelian, a fact which (in general) depends on rather special properties of the unitary groups.
In low dimensions, this result has the following manifestation:

\begin{corollary}$\label{model}$ Let $M$ be either the circle or an aspherical compact surface.  Then there is a weak equivalence between $\K(\pi_1 (M))$ and
the space 
$$ \hocolim_{\stackrel{\maps}{\oplus 1}} ( \Rep(\pi_1 M)_{hU} ) := \hocolim \left( \Rep(\pi_1 M)_{hU} \stackrel{\oplus 1}{\maps} \Rep(\pi_1 M)_{hU} 
\stackrel{\oplus 1}{\maps} \cdots \right)$$
where $\oplus 1$ denotes the map induced by block sum with the identity matrix $1\in U(1)$.
\end{corollary}

There are at least two ways to show that $\Rep(\pi_1 M)$ is stably group-like with respect to $1\in \Hom(\pi_1 M, U(1))$.  In Corollaries~\ref{rep-ctd} and~\ref{rep-ctd-no} we use Yang--Mills theory to show that $\Rep(\pi_1 M)$ is stably group-like for any compact, aspherical surface $M$.  (In the orientable case, this amounts to showing that the representation spaces are all connected, which is a well-known folk theorem.)  This argument is quite close to Ho and Liu's proof of connectivity for the moduli space of flat connections~\cite[Theorem 5.4]{Ho-Liu-non-orient}.  
For most surfaces, other work of Ho and Liu~\cite{Ho-Liu-ctd-comp-II} gives an alternative method, depending on Alekseev, Malkin, and Meinrenken's theory of quasi-Hamiltonian moment maps~\cite{AMM}.
A version of that argument, adapted to the present situation, appears in the author's thesis~\cite[Chapter 6]{Ramras-thesis}.


\section{Representations and flat connections}$\label{rep-flat-section}$

Let $M$ denote an $n$--dimensional, compact, connected manifold, with a fixed basepoint $m_0\in M$.  
Let $G$ be a compact Lie group, and $P\stackrel{\pi}{\to} M$ be a smooth principal $G$--bundle, with a fixed basepoint $p_0\in \pi^{-1}(m_0) \subset P$.  Our principal bundles will always have a $\emph{right}$ action of the structure group $G$.
In this section we explain how to pass from $G$--representation spaces of $\pi_1(M)$ to spaces of flat connections on principal $G$--bundles over $M$, which form critical sets for the Yang--Mills functional.  
The main result of this section is the following proposition, which we state informally for the moment.

\medskip
{\bf Proposition \ref{rep-flat}}\quad
{\sl For any $n$--manifold $M$ and any compact, connected Lie group $G$, holonomy induces a $G$--equivariant homeomorphism
$$\coprod_{[P_i]} \flatc(P_i)/\G_0 (P_i) \stackrel{\overline{\mH}}{\maps} \Hom(\pi_1(M), G),$$
where the disjoint union is taken over some set of representatives for the (unbased) isomorphism classes of principal $G$--bundles over $M$.  (Note that to define $\overline{\mH}$ we choose, arbitrarily, a base point in each representative bundle $P_i$.)
}
\medskip

Here $\G_0 (P)$ denotes the based gauge group, and consists of all principal bundle automorphisms of 
$P$ that restrict to the identity on the fiber over $m_0\in M$.

\subsection{The one-dimensional case}$\label{S^1}$
Before beginning the proof of Proposition~\ref{rep-flat}, we explain how this result immediately leads to an analogue of the Atiyah--Segal theorem for the infinite cyclic group $\bbZ$.  This argument will motivate our approach for surface groups.

By Corollary~\ref{model}, deformation $K$--theory of $\Z$ is built from the homotopy
orbit spaces 
$$(U(n)^{\mathrm{Ad}})_{hU(n)} :=EU(n)\cross_{U(n)} \Hom(\Z, U(n)),$$
and the homotopy groups of $LBU(n) = \Map(S^1, BU(n))$ are just the complex $K$-groups of $S^1 = B\Z$ (in dimensions $0<*<2n$).  Thus the well-known homotopy equivalence
\begin{equation}\label{ad-EG}
(U(n)^{\mathrm{Ad}})_{hU(n)} \heq LBU(n)
\end{equation}
may be interpreted as an Atiyah--Segal theorem for the group $\Z$, and upon taking colimits (\ref{ad-EG}) yields an isomorphism $\K_*(\bbZ) \isom K^{-*}(S^1)$ for any $*\geqs 0$. 
(The equivalence (\ref{ad-EG}) is well-known for any group $G$, but the only general reference of which I am aware is the elegant proof given by K. Gruher in her thesis~\cite{Gruher}).)  

Proposition~\ref{rep-flat} actually leads to a proof of (\ref{ad-EG}) for any compact, connected Lie group $G$.  Connections $A$ over the circle are always flat, and the $\G_0 (P)$ acts freely, so by Proposition~\ref{rep-flat} and a basic fact about homotopy orbit spaces we have
$$\Hom(\Z, G)_{hG} \homeo \left(\A (S^1\cross G)/ \Map_* (S^1, G)\right)_{hG} \heq \left(\A (S^1\cross G)\right)_{h \Map(S^1, G)}.$$
But connections form a contractible (affine) space, so the right hand side is the classifying space of the full gauge group.  By Atiyah and Bott~\cite[Section 2]{A-B}, $\Map (S^1, BG) = LBG$ is a model for $B\Map(S^1, G)$, so $(G^{\mathrm{Ad}})_{hG} \heq LBG$ as desired.

When $\Z$ is replaced by the fundamental group of a two-dimensional surface, one can try to mimic this argument.  Not all connections are flat in this case, but flat connections do form a critical set for the Yang--Mills functional $L\co \A \to \bbR$.  In Section~\ref{Harder-Narasimhan}, we will use Morse theory for $L$ to prove a connectivity result for the space of flat connections.

\subsection{Sobolev spaces of connections and the holonomy map}$\label{rep-flat-sec}$

In order to give a precise statement and proof of Proposition~\ref{rep-flat}, we need to introduce the relevant Sobolev spaces of connections and gauge transformations.  Our notation and discussion follow Atiyah--Bott~\cite[Section 14]{A-B}, and another excellent reference is the appendix to Wehrheim~\cite{Wehrheim}.

 We use the notation $L^p_k$ to denote functions with $k$ weak (i.e. distributional) derivatives, each in the Sobolev space $L^p$.
We will record the necessary assumptions on $k$ and $p$ as they arise.  The reader interested only in the applications to deformation $K$--theory may safely ignore these issues, noting only that all the results of this section hold in the Hilbert space $L^2_k$ for large enough $k$.  When $n=2$, our main case of interest, we just need $k\geqs 2$.

\begin{definition}$\label{Sobolev}$ Let $k\geqs 1$ be an integer, and let $1\leq p < \infty$.  We denote the space of all connections on the bundle $P$ of Sobolev class $L^p_k$ by $\A^{k, p}(P)$.  This is an affine space, modeled on the Banach space of $L^p_k$ sections of the vector bundle $T^* M \otimes \mathrm{ad \,} P$ (here $\mathrm{ad \,} P = P\cross_G \mathfrak g$, and $\mathfrak g$ is the Lie algebra of $G$ equipped with the adjoint action).  Hence $\A^{k,p}(P)$ acquires a canonical topology, making it homeomorphic to the Banach space on which it is modeled.  Flat $L^p_k$ connections are defined to be those with zero curvature.  The subspace of flat connections is denoted by $\flatc^{k,p} (P)$.

We let $\G^{k+1,p}(P)$ denote the gauge group of all bundle automorphisms of $P$ of class $L^p_ {k+1}$, and (when $(k+1)p > n$) we let $\G ^{k+1, p}_0(P)$ denote the subgroup of based automorphisms (those which are the identity on the fiber over $m_0\in M$).  These gauge groups are Banach Lie groups, and act smoothly on $\A^{k,p}(P)$.  We will always use the left action, meaning that we let gauge transformations act on connections by pushforward.  We denote the group of all continuous gauge transformations by $\G(P)$.  Note that so long as $(k+1)p >n$, the Sobolev Embedding Theorem gives a continuous inclusion $\G^{k+1,p}(P)\injects \G(P)$, and hence in this range $\G ^{k+1, p}_0(P)$ is well-defined.
We denote the smooth versions of these objects by $(-)^{\infty} (P)$.
\end{definition}

The following lemma is well-known.

\begin{lemma}$\label{smoothing}$ For $(k+1)p > n$, the inclusion $\G^{k+1,p}(P) \injects \G(P)$ is a weak equivalence.
\end{lemma}
\begin{proof} Gauge transformations are simply sections of the adjoint bundle $P\cross_G \mathrm{Ad} (G)$ (see~\cite[Section 2]{A-B}).  Hence this result follows from general approximation results for sections of smooth fiber bundles.
\end{proof}

The continuous inclusion $\G^{k+1,p} (P) \injects \G(P)$ implies that there is a well-defined, continuous homomorphism $r\co\G^{k+1,p}(P) \to G$ given by restricting a gauge transformation to the fiber over the basepoint $m_0\in M$.  To be precise, $r(\phi)$ is defined by $p_0 \cdot r(\phi) = \phi(p_0)$, and hence depends on our choice of basepoint $p_0\in P$.

\begin{lemma}$\label{restriction}$ Assume $G$ is connected and $(k+1)p > n$.  Then the restriction map 
$r\co\G^{k+1,p}(P) \maps G$ induces a homeomorphism 
$\bar{r}\co \G^{k+1,p}(P)/\G^{k+1,p}_0 (P) \stackrel{\homeo}{\maps} G$.  
The same statements hold for the smooth gauge groups.
\end{lemma}
\begin{proof}  Thinking of gauge transformations as sections of the adjoint bundle, we may deform the identity map $P\to P$ over a neighborhood of $m_0$ so that it takes any desired value at $p_0$ (here we use connectivity of $G$).  Hence $r$ and $\bar{r}$ are surjective.

By a similar argument, we may construct continuous local sections $s\co U\to \G^{\infty}(P)$ of the map $r$, where $U\subset G$ is any chart.  If $\pi\co \G^{\infty}(P) \to \G(P)^{\infty}/\G^{\infty}_0 (P)$ is the quotient map, then the maps $\pi \circ s$ are inverse to $\overline{r}$ on $U$.  Hence $\overline{r}^{-1}$ is continuous.  The same argument applies to $\G^{k+1,p} (P)$. 
\end{proof}

I do not know whether Lemma~\ref{restriction} holds for non-connected groups; certainly the proof shows that the image of the restriction map is always a union of components.

Flat connections are related to representations of $\pi_1 M$ via the holonomy map.  Our next goal is to analyze this map carefully in the current context of Sobolev connections.  
The holonomy of a smooth connection is defined via parallel transport: given a smooth loop $\gamma$ based at $m_0\in M$, there is a unique $A$--horizontal lift $\widetilde{\gamma}$ of $\gamma$ with $\widetilde{\gamma} (0) = p_0$, and the holonomy representation $\mcH(A) = \rho_A$  is then defined by the equation 
$\widetilde{\gamma}(1) \cdot \rho_A ([\gamma]) = p_0.$
(Since flat connections are locally trivial, a standard compactness argument shows that this definition depends only on the homotopy class $[\gamma]$ of $\gamma$.)  It is important to note here that the holonomy map depends on the chosen basepoint $p_0\in P$.
Further details on holonomy appear in the Appendix.

\begin{lemma}$\label{holonomy-cont}$  The holonomy map $\flatc^{k,p} (P) \to \Hom(\pi_1 M, G)$ is continuous if $k\geqs 2$ and $(k-1)p>n$.
\end{lemma}
\begin{proof}  The assumptions on $k$ and $p$ guarantee a continuous embedding $L^p_k (M) \injects C^1(M)$.  Hence if $A_i \in \fc^{k,p} (P)$ is a sequence of connections converging (in $\fc^{k,p} (P)$) to $A$, then $A_i\to A$ in $C^1$ as well.  We must show that for any such sequence, the holonomies of the $A_i$ converge to the holonomy of $A$.  

It suffices to check that for each loop $\gamma$ the holonomies around $\gamma$ converge.  These holonomies are defined (continuously) in terms of the integral curves of the vector fields $V (A_i)$ on 
$\gamma^* P$ arising from the connections $A_i$.  Since these vector fields converge in the $C^1$ norm, we may assume that the sequence $||V(A_i) - V(A)||_{C^1}$ is decreasing and less than 1.  By interpolating linearly between the $V (A_i)$, we obtain a vector field on $\gamma^*P \cross I$ which at time $t_i$ is just $V (A_i)$, and at time $0$ is $V(A)$.  This is a Lipschitz vector field and hence its integral curves vary continuously in the initial point (Lang~\cite[Chapter IV]{Lang-dg}) completing the proof.
\end{proof}

\begin{remark}\label{better-cont}  With a bit more care, one can prove Lemma~\ref{holonomy-cont} under the weaker assumptions $k\geqs 1$ and $kp>n$.  The basic point is that these assumptions give an embedding $L^p_k(M) \injects C^0 (M)$, and by compactness $C^0(M)\injects L^1(M)$ (and similarly after restricting to a smooth curve in $M$).  Working in local coordinates, one can deduce continuity of the holonomy map from the fact that limits commute with integrals in $L^1([0,1])$.
\end{remark}

\begin{lemma}$\label{orbits}$ 
Assume $p> n/2$ (and if $n=2$, assume $p\geqs 4/3$).  If $G$ is connected, then each $\G^{k+1,p}_0(P)$--orbit in $\flatc^{k,p}(P)$ contains a unique $\G^{\infty}_0 (P)$--orbit of smooth connections.  
\end{lemma}
\begin{proof} By Wehrheim~\cite[Theorem 9.4]{Wehrheim}, the assumptions on $k$ and $p$ guarantee that each $\G^{k+1,p}(n)$ orbit in $\fc^{k,p}(n)$ contains a smooth connection.  Now, say $\phi \cdot A$ is smooth for some
$\phi\in \G^{k+1,p} (P)$.  
By Lemma~\ref{restriction}, there exists a smooth gauge transformation $\psi$ such that $r(\psi) = r(\phi)^{-1}$.  The composition $\psi \circ \phi$ is based, and since $\psi$ is smooth we know that 
$(\psi \circ \phi) \cdot A$ is smooth.  This proves existence.  For uniqueness, say $\phi\cdot A$ and $\psi\cdot A$ are both smooth, where $\phi, \psi\in \G^{k+1,p}_0 (P)$.  Then $\phi \psi^{-1}$ is smooth by~\cite[Lemma 14.9]{A-B}, so these connections lie in the same $\G^{\infty}_0$--orbit.
\end{proof}

The following elementary lemma provides some of the compactness we will need.

\begin{lemma}\label{finiteness}  If $G$ is a compact Lie group, then only finitely many isomorphism classes of principal $G$--bundles over $M$ admit flat connections.  \end{lemma}

\begin{proof} 
As described in the Appendix, any bundle $E$ admitting a flat connection $A$ is isomorphic to the bundle induced by holonomy representation $\rho_A\co \pi_1 M \to G$.  If two bundles $E_0$ and $E_1$ arise from representations $\rho_0$ and $\rho_1$ in the same path component of $\Hom(\pi_1 M, G)$, then choose a path $\rho_t$ of representations connecting $\rho_0$ to $\rho_1$.  The bundle
$$E = (\wt{M}\cross [0,1] \cross G) \Big/ (\wt{m}, t, g) \sim (\wt{m}\cdot \gamma, t, \rho_t(\gamma)^{-1}g)$$
 is a principal $G$--bundle over $M \cross [0,1]$ and provides a bundle homotopy between $E_0$ and $E_1$;  by the Bundle Homotopy Theorem we conclude $E_0\isom E_1$.  Hence the number of isomorphism classes admitting flat connections is at most the number of path components of $\Hom(\pi_1 M, G)$.
 
Now recall that any compact Lie group is in fact algebraic: its image under a faithful representation $\rho\co G\to \mathrm{GL}_n \bbC$ is Zariski closed (see Sepanski~\cite[Excercise 7.35, p. 186]{Sepanski} for a short proof using the Stone-Weierstrass Theorem and Haar measure).  Since $\pi_1 M$ is finitely generated (by $k$ elements, say), $\Hom(\pi_1 M, G)$ is the subvariety of $G^k$ cut out by the relations in $\pi_1 M$.  So this space is a real algebraic variety as well, hence triangulable (see Hironaka~\cite{Hironaka}).  Since compact CW complexes have finitely many path components, the proof is complete.
\end{proof}

\begin{remark} We note that the previous lemma also holds for non-compact algebraic Lie groups, by a result of Whitney~\cite{Whitney} regarding components of varieties.
\end{remark}

We can now prove the result connecting representations to Yang--Mills theory.

\begin{proposition}$\label{rep-flat}$ Assume $p> n/2$ (and if $n=2$, assume $p > 4/3$), $k\geqs 1$, and $kp>n$.  Then for any $n$--manifold $M$ and any compact, connected Lie group $G$, the holonomy map induces a $G$--equivariant homeomorphism
$$\coprod_{[P_i]} \flatc^{k,p}(P_i)/\G^{k+1,p}_0 (P_i) \stackrel{\overline{\mH}}{\maps} \Hom(\pi_1(M), G),$$
where the disjoint union is taken over some set of representatives for the (unbased) isomorphism classes of principal $G$--bundles over $M$.  (Note that to define $\overline{\mH}$ we choose, arbitrarily, a base point in each representative bundle $P_i$.)

The $G$--action on the left is induced by the actions of $\G^{k+1,p}(P_i)$ together with the homeomorphisms $\G^{k+1,p}(P_i)/\G^{k+1,p}_0(P_i) \homeo G$, which again depends on the chosen basepoints in the bundles $P_i$.
\end{proposition}
\begin{proof}  The assumptions on $k$ and $p$ allow us to employ all previous results in this section (note Remark~\ref{better-cont}).
It is well-known that the holonomy map 
$$H\co\coprod_{[P_i]} \flatc^{\infty} (P_i) \maps \Hom(\pi_1 (M), G)$$
is invariant under the action of the based gauge group and induces an equivariant bijection
$$\bar{H}\co \coprod_{[P_i]} \flatc^{\infty} (P_i)/\G^{\infty}_0 (P_i) \maps \Hom(\pi_1(M), G).$$
For completeness we have included a proof of this result in the Appendix.
By Lemma~\ref{orbits}, the left hand side is unchanged (set-theoretically) if we replace
$\flatc^{\infty}$ and $\G^{\infty}_0$ by $\flatc^{k,p}$ and $\G^{k+1,p}_0$, and hence Lemma~\ref{holonomy-cont}
tells us that we have a continuous equivariant bijection
$$\bar{\mH}\co \coprod_{[P_i]} \flatc^{k,p} (P_i)/\G^{k+1,p}_0 (P_i) \maps \Hom(\pi_1(M), G).$$

We will show that for each $P$, $\flatc^{k,p}(P)/\G^{k+1,p}_0(P)$ is sequentially compact.  Since, by Lemma~\ref{finiteness}, 
 only finitely many isomorphism types of principal $G$--bundle admit flat connections, this implies that
$$\coprod_{[P_i]} \flatc^{k,p}(P_i)/\G^{k+1,p}_0 (P_i)$$
is sequentially compact.
A continuous bijection from a sequentially compact space to a Hausdorff space is a homeomorphism, so this will complete the proof.   

The Strong Uhlenbeck Compactness Theorem~\cite{Wehrheim} (see also Daskalopoulos~\cite[Proposition 4.1]{Dask}) states that the space $\flatc^{k,p}(P)/\G^{k+1,p}(P)$ is sequentially compact.  Given a sequence $\{A_i\}$ in $\flatc^{k,p}(P)$,  there exists a sub-sequence $\{A_{i_j}\}$ and a sequence $\phi_j\in \G^{k+1,p}(n)$ such that $\phi_j\cdot A_{i_j}$ converges in $\A^{k,p}$ to a flat connection $A$.  Let $g_j = r(\phi_j)$.  Since $G$ is compact, passing to a sub-sequence if necessary we may assume that the $g_j$ converge to some $g\in G$.
The proof of Lemma~\ref{restriction} shows that we may choose a convergent sequence $\psi_j\in\G^{k+1,p} (P)$ with $r(\psi_j) = g_j^{-1}$; we let $\psi = \lim \psi_j$, so $r(\psi) = g^{-1}$.
By continuity of the action, the sequence $(\psi_j \circ \phi_j) \cdot A_{i_j}$ converges to 
$\psi\cdot A$.  Since $\psi_j \circ \phi_j \in \G^{k+1,p}_0(P)$, this completes the proof. 
\end{proof}

\begin{remark} Since $\Hom(\pi_1 M, G)$ is compact, Proposition~\ref{rep-flat} implies compactness of $\fc^{k,p}(P)/\G^{k+1,p}_0 (P)$.  
However, point-set considerations alone show that sequential compactness of $\fc^{k,p}(P)/\G^{k+1,p}_0 (P)$ suffices to prove its compactness: specifically, $\fc^{k,p}(P)$ is second countable, since it is a subspace of a separable Banach space.  The quotient map of a group action is open, so $\fc^{k,p}(P)/\G^{k+1,p}_0 (P)$ is second countable as well.  Now, any second countably, sequentially compact space is compact.  (The necessary point-set topology can be found in Wilansky~\cite[5.3.2, 7.3.1, 5.4.1]{Wilansky}.)

More interesting is that Proposition~\ref{rep-flat} implies that the based gauge orbits in $\fc^{k,p}(P)$ are closed (the quotient embeds in $\Hom(\pi_1 M, G)$).  Since $G$ is compact, one also concludes that the full gauge orbits are closed.
\end{remark}


\newpage

\section{The
Harder-Narasimhan stratification and Morse theory for the Yang--Mills Functional }$\label{Harder-Narasimhan}$

In Section~\ref{rep-flat-section}, we explained how to pass from spaces of representations to spaces of flat connections.  We now focus on the case where $M$ is a compact surface and $G= U(n)$. 
We wish to compare the space of flat connections on a $U(n)$--bundle $P$ over $M$ to the \emph{contractible} space of all connections on $P$, and in particular we want to understand what happens as the rank of $P$ tends to infinity.  Atiyah and Bott made such a comparison (for a fixed bundle $P$) using computations in equivariant cohomology.  We will work directly with homotopy groups, using Smale's infinite dimensional transversality theorem.
A (co)homological approach could be used in the orientable case (details may be found in~\cite{Ramras-YM}), but there are difficulties, related to equivariant Thom isomorphisms, in extending such an argument to the non-orientable case.  These issues are the focus of ongoing work with Ho and Liu. 

The main result of this section is the following connectivity estimate.

\medskip
\noindent{\bf Proposition \ref{connectivity}}\quad
{\sl Let $M = M^g$ denote a compact Riemann surface of genus $g$, and let $n>1$ be an integer.
Then the space $\flatc^k(n)$ of flat connections on a trivial rank $n$ bundle over $M$ is $2g(n-1)$--connected, and if $\Sigma$ is a non-orientable surface with oriented double cover $M^g$, then the space of flat connections on any principal $U(n)$--bundle over $\Sigma$ is $(g(n-1) -1)$--connected.
}
\medskip

In the orientable case this result is in fact sharp; in the non-orientable case it can be improved significantly~\cite{Ramras-YM}.

We will work in the Hilbert space of $L^2_k$ connections, and we assume $k\geqs 2$ so that the results of Section~\ref{rep-flat-sec} apply.  We suppress $p=2$ from the notation, writing simply $\A^k$, $\G^k$, and so on.  Over a Riemann surface, any principal $U(n)$--bundle admitting a flat connection is trivial (see Corollary~\ref{rep-ctd}), and hence we restrict our attention to the case $P = M\cross U(n)$ and write $\A^k(n) = \A^k(M\cross U(n))$, etc.

For any smooth principal $U(n)$--bundle $P\to M$, the Yang--Mills functional 
$$L\co\A^k(P)\to \bbR$$
is defined by the formula
$$L(A) = \int_M ||F(A)||^2 d\mathrm{vol}$$
where $F(A)$ denotes the curvature form of the connection $A$ and the volume of $M$ is normalized to be $1$.  The space $\fc^k(n)$ of flat connections forms a critical set for the $L$~\cite[Proposition 4.6]{A-B}, and so one hopes to employ Morse-theoretic ideas to compare the topology of this critical set to the topology of $\A^k(P)$.  In particular, the gradient flow of $L$ should allow one to define stable manifolds associated to critical sets of $M$, which should deformation retract to those critical subsets.
The necessary analytical work has been done by Daskalopoulos~\cite{Dask} and R{\aa}de~\cite{Rade}, and furthermore Daskalopoulos has explicitly identified the Morse stratification of $\A^k(P)$ (proving a conjecture of Atiyah and Bott).  We now explain this situation.  

We now recall (see~\cite[Sections 5, 7]{A-B}) that there is a bijective correspondence between connections on a principal $U(n)$--bundle $P$ and Hermitian connections on the associated Hermitian vector bundle $E = P\cross_{U(n)} \bbC^n$.  
When the base manifold is a Riemann surface, the latter space may in fact be viewed as the space of holomorphic structures on $E$: first, there is a bijective correspondence between holomorphic structures and $(0,1)$--connection $\bar{\partial}\co \Omega^0_{M}(E)\to \Omega^{0,1}_M (E)$, provided by the fact that each $(0,1)$--connection on a complex curve is integrable (see Donaldson--Kronheimer~\cite[Section 2.2.2]{DK}), together with the fact that a holomorphic structure is determined by its sheaf of holomorphic sections (sections $s\co M\to E$ with $\bar{\partial} s = 0$).  Now, each Hermitian connection on $E$ has an associated $(0,1)$--connection, and by Griffiths--Harris~\cite[p. 73]{Griffiths-Harris} this in fact induces an isomorphism of affine spaces.  

From now on we will view holomorphic structures in terms of their associated $(0,1)$--connections.
Since the space of $(0,1)$--connections is an affine space modeled the vector space of (smooth) sections of the vector bundle $\Omega_M^{0,1} (\mathrm{End} (E))$ of endomorphism-valued $(0,1)$--forms, we may define Sobolev spaces $\mcC^k(E) = \mcC^{k,2}(E)$ of $(0,1)$--connections simply by taking $L^2_k$--sections of this bundle. When $M$ is a Riemann surface, the above isomorphism of affine spaces extends to an isomorphism $\A^k(P) \isom \mcC^k(P\cross_{U(n)} \bbC^n)$.

There is a natural algebraic stratification of $\mcC^k(E)$ called the Harder-Narasimhan stratification, which turns out to agree with the Morse stratification of $\A^k (P)$.  We now describe this stratification in the case $E = M\cross \bbC^n$.

\begin{definition} Let $E$ be a holomorphic bundle over $M$.  Let $\mathrm{deg}(E)$ denote its first Chern number and let $\mathrm{rk} (E)$ denote its dimension.  We call $E$ semi-stable if for every proper holomorphic sub-bundle $E'\subset E$, one has 
$$\frac{\mathrm{deg}(E')}{\mathrm{rk} (E')} \leqs \frac{\mathrm{deg}(E)}{\mathrm{rk}(E)}.$$
Replacing the $\leqs$ by $<$ in this definition, one has the definition of a \emph{stable} bundle.
\end{definition}

Given a holomorphic structure $\E$ on the bundle $M\cross \bbC^n$, there is a unique filtration (the Harder-Narasimhan filtration~\cite{H-N})
$$0 = \E_0 \subset \E_1 \subset \cdots \E_r = \E$$
of $\E$ by holomorphic sub-bundles with the property that each quotient $\mcD_i = \E_i/\E_{i-1}$ is semi-stable ($i = 1,\ldots ,r$) and $\mu(\mcD_1) > \mu(\mcD_2) > \cdots > \mu(\mcD_r)$, where 
$\mu(\mcD_i) = \frac{\mathrm{deg} (\mcD_i)}{\mathrm{rank} (\mcD_i)}$, and $\mathrm{deg} (\mcD_i)$ is the first Chern number of the vector bundle $\mcD_i$.
Letting $n_i =$ rank$(D_i)$ and $k_i =$ deg$(D_i)$, we call the sequence 
$$\mu = ((n_1, k_1), \ldots, (n_r, k_r))$$
the \emph{type} of $\E$.  Since ranks and degrees add in exact sequences, we have $\sum_i n_i = n$ and $\sum_i k_i = 0$.
By~\cite[Section 14]{A-B}, each orbit of the \emph{complex} gauge group on $\mcC^k(E)$ contains a unique isomorphism type of smooth $(0,1)$--connections (i.e. holomorphic structures), so we may define $\mcC^k_{\mu} = \mcC^k_{\mu}(n) \subset \mcC^k(n)$ to be the subspace of all $(0,1)$--connections gauge-equivalent to a smooth connection of type $\mu$, and the $\mcC^k_{\mu}$ partition $\mcC^k(n)$.
Note that the semi-stable stratum corresponds to $\mu = ((n,0))$.

It is a basic fact that every flat connection on $E$ corresponds to a semi-stable bundle: the Narasimhan-Seshardri Theorem~\cite[(8.1)]{A-B} says that irreducible representations induce \emph{stable} bundles.
By Proposition~\ref{rep-flat}, every flat connection comes from some unitary representation, which is a sum of irreducible representations, and hence the holomorphic bundle associated to any representation, i.e. any flat connection, is a sum of stable bundles.  Finally, an extension of stable bundles of the same degree is always semi-stable.  

We can now state the result we will need.

\begin{theorem}[Daskalopoulos, Rade]$\label{retraction}$ Let $M$ be a compact Riemann surface.  
Then the gradient flow of the Yang--Mills functional is well-defined for all positive time, and defines continuous deformation retractions from the Harder-Narasimhan strata $\mcC^k_\mu$ to their critical subsets.  Moreover, these strata are locally closed submanifolds of $\mcC^k(n)$ of complex codimension
$$c(\mu) = \left( \sum_{i>j} n_i k_j - n_j k_i \right) + (g-1) \left( \sum_{i>j} n_i n_j  \right).$$

In particular, there is a continuous deformation retraction (defined by the gradient flow of $L$) from the space $\Css^k(n)$ of all semi-stable $L^2_k$ $(0,1)$--connections on $M\cross \bbC^n$ to the subspace $\flatc^k(n)$ of flat (unitary) connections.
\end{theorem}

\begin{remark} This result holds for any $C^{\infty}$ vector bundle.  Daskalopoulos proved convergence of the Yang--Mills flow modulo gauge transformations, and established continuity in the limit on (the gauge quotient of) each Harder-Narasimhan stratum (which he proved to be submanifolds).  R{\aa}de later proved the full convergence result stated above.  We will discuss the analogue of this situation in the non-orientable case in the proof of Proposition~\ref{connectivity}.  One could ask for a result slightly stronger than Theorem~\ref{retraction}: since the gradient flow of $L$ converges at time $+\infty$ to give a continuous retraction from each $\mcC^k_\mu$ to its critical set, this stratum is a disjoint union of Morse strata.  However, I do not know in general whether $\mcC^k_\mu$ is connected.  Hence the Morse stratification may be finer than the Harder-Narasimhan stratification.
\end{remark}

The following definition will be useful.

\begin{definition}$\label{admissible}$ Consider a sequence of pairs of integers 
$((n_1, k_1), \ldots, (n_r, k_r))$.  We call such a sequence \emph{admissible} of total rank $n$ (and total Chern class $0$) if 
$n_i > 0$ for each $i$, $\sum n_i = n$, $\sum_i k_i = 0$, and $\frac{k_1}{n_1}>\cdots>\frac{k_r}{n_r}$.    Hence admissible sequences of total rank $n$ and total Chern class $0$ are precisely those describing Harder-Narasimhan strata in $\mcC(n)$.
We denote the collection of all admissible sequences of total rank $n$ and total Chern class $0$ by $\mathcal{I}(n)$.  
\end{definition}

We now compute the minimum codimension of a non semi-stable stratum.  In particular, this computation shows that this minimum tends to infinity with $n$, so long as the genus $g$ is positive.  This result has been extended to the case of non-trivial bundles in~\cite{Ramras-YM}.

\begin{lemma}$\label{codim}$ The minimum (real) codimension of a non semi-stable stratum in $\mcC^k(n)$ ($n>1$) is precisely $2n + 2(n-1)(g-1) = 2g(n-1) +2$.
\end{lemma}
\begin{proof}  Let $\mu = ((n_1, k_1), \ldots, (n_r, k_r)) \in \I(n)$ be any admissible sequence with $r>1$.
Then from Theorem~\ref{retraction}, we see that it will suffice to show that
\begin{equation}\label{term1} 
\sum_{i>j} n_i k_j - n_j k_i \geqs n
\end{equation}
and
\begin{equation}\label{term2} 
\sum_{i>j} n_i n_j \geqs n-1.
\end{equation}

To prove (\ref{term1}), we begin by noting that since $\sum k_i = 0$ and the ratios $\frac{k_i}{n_i}$ are strictly decreasing, we must have $k_1>0$ and $k_n<0$.  Moreover, there is some $l_0\in \bbR$ such that $k_l \geqs 1$ for $l<l_0$ and $k_l\leqs -1$ for $l>l_0$.  We allow $l_0$ to be an integer if and only if $k_l = 0$ for some $l$; then this integer $l$ is unique, and in this case we define $l_0 := l$.  Since $r\geqs 2$, we know that $1<l_0<r$.

Now, if $i>l_0>j$ we have $k_j\geqs 1$ and $k_i\leqs -1$, so, 
$$n_ik_j - n_j k_i \geqs n_i + n_j.$$
If $i>l_0$ and $j=l_0$, we have $k_j = 0$ and $k_i \leqs -1$, so
$$n_i k_{l_0} - n_{l_0} k_i \geqs 0 + n_{l_0} = n_{l_0}.$$
Finally, if $i = l_0$ and $j<l_0$, then $k_i= 0$ and $k_j \geqs 1$ so we have
$$n_{l_0} k_j - n_j k_{l_0} \geqs n_{l_0} - 0 = n_{l_0}.$$

Now, since $n_i k_j - n_j k_i = n_in_j (k_j/n_j - k_i/n_i)$ and the $k_l/n_l$ are strictly decreasing, we know that each term in the sum $\sum_{i>j} n_i k_j - n_j k_i$ is positive.  Dropping terms and applying the above bounds gives
$$\sum_{i>j} n_i k_j - n_j k_i 
\geqs \sum_{i>l_0>j} (n_i k_j - n_j k_i)
	+ \sum_{ l_0>j} (n_{l_0} k_j - n_j k_{l_0})
	+ \sum_{i>l_0} (n_i k_{l_0} - n_{l_0} k_i)$$
$$\geqs \sum_{i>l_0>j} (n_i + n_j) + \sum_{l_0>j} n_{l_0} + \sum_{i> l_0} n_{l_0}.$$
(In the second and third expressions, the latter sums are taken to be empty if $l_0$ is not an integer.)
Since $\sum n_i = n$, to check that the above expression is at least $n$ it suffices to check that each $n_i$ appears in the final sum.  But since $1<l_0<r$, each $n_l$ with $l\neq l_0$ appears in the first term, and if $l_0\in \bbN$ then $n_{l_0}$ appears in both of the latter terms.
This completes the proof of (\ref{term1}).

To prove (\ref{term2}), we fix $r\in \bbN$ ($r\geqs 2$) and consider partitions $\vect{p} = (p_1, \ldots, p_r)$ of $n$.  We will minimize the function $\phi_r(\vect{p}) = \sum_{i>j} p_i p_j$, over all length $r$ partitions of $n$. 

Consider a partition $\vect{p} = (p_1, \ldots, p_r)$ with $p_m \geqs p_l >1$ ($l\neq m$), and define another partition $\vect{p'}$ by setting
$$p_i' = \left\{ \begin{array}{ll}
					p_i, &  i\neq l, \,\, m,\\
					p_l  - 1, & i = l, \\
					p_m + 1 & i = m.
				      \end{array}
					\right.
$$  
It is easily checked that $\phi_r (\vect{p}) > \phi_r (\vect{p'})$.

Now, if we start with any partition $\vect{p}$ such $p_i >1$  for more than one index $i$, the above argument shows that $\vect{p}$ cannot minimize $\phi_r$.  Thus $\phi_r$ is minimized by the partition $\vect{p_0} = (1, \ldots, 1, n-r-1)$, and $\phi_r (\vect{p_0}) = \binom{r-1}{2} + (r-1)(n-r-1)$.
The latter is an increasing function for $r\in(0,n)$ and hence $\sum_{i>j} p_i p_j$ is minimized by the partition $(1, n-1)$.  This completes the proof of (\ref{term2}).

Finally, note that the sequence $((1,1), (n-1, -1))$ has complex codimension $n + (n-1)(g-1)$. 
\end{proof}

\begin{remark} It is interesting to note that the results in the next section clearly fail in the case when $M$ has genus 0.  From the point of view of homotopy theory, the problem is that $S^2$ is not the classifying space of its fundamental group, and so one should not expect a relationship between $K$--theory of $S^2$ and representations of $\pi_1 S^2 = 0$.  But the only place where our argument breaks down is the previous lemma, which tells us that there are strata of complex codimension 1 in the Harder-Narasimhan stratification of 
$\mcC^k(S^2\cross \bbC(n))$, and in particular the minimum codimension does not tend to infinity with the rank. 
\end{remark}

The main result of this section will be an application of the following infinite-dimensional transversality theorem, due to Smale~\cite[Theorem 19.1]{Abraham-Robbin} (see also~\cite{Abraham-Smale}).  Recall that a residual set in a topological space is a countable intersection of open, dense sets.  By the Baire category theorem, any residual subset of a Banach space is dense, and since any Banach manifold is \emph{locally} a Banach space, any residual subset of a Banach manifold is dense as well.

\begin{theorem}[Smale]$\label{transversality}$
Let $A$, $X$, and $Y$ be second countable $C^r$ Banach manifolds, with $X$ of finite dimension $k$.  Let $W\subset Y$ be a (locally closed) submanifold of $Y$, of finite codimension $q$.
Assume that $r > \max (0, k-q)$.
Let
$\rho\co A\to C^r(X,Y)$ be a $C^r$--representation, that is, a function for which the evaluation map 
$\mathrm{ev}_{\rho}\co A\cross X \to Y$ given by $\mathrm{ev}_{\rho} (a, x) = \rho(a) x$ is of class $C^r$.

For $a\in A$, let $\rho_a \co X\to Y$ be the map $\rho_a (x) = \rho(a) x$.  Then 
$\{a\in A | \rho_a \pitchfork W\}$ is residual in $A$, provided that $\mathrm{ev}_{\rho} \pitchfork W$.
\end{theorem}

\begin{corollary}$\label{stratification}$
Let $Y$ be a second countable Banach space, and let $\{W_i\}_{i\in I}$ be a countable collection of (locally closed) submanifolds of $Y$ with finite codimension.  Then if $U = Y - \bigcup_{i\in I} W_i$ is a non-empty open set, it has connectivity at least $\mu -2$, where
$$\mu = \min \{\mathrm{codim} \,\,W_i \,\,:\,\, i\in I\}.$$
\end{corollary}
\begin{proof}  
To begin, consider a continuous map $f\co S^{k-1} \to U$, with $k-1 \leqs \mu - 2$.  We must show that $f$ is null-homotopic in $U$; note that our homotopy need not be based.  First we note that since $U$ is open, $f$ may be smoothed, i.e. we may replace $f$ by a $C^{k+1}$ map $f'\co S^{k-1} \to U$ which is homotopic to $f$ inside $U$ (this follows, for example, from Kurzweil's approximation theorem~\cite{Kurzweil}). 

Choose a smooth function $\phi\co\bbR\to \bbR$ with the property that $\phi (t) = 1$ for $t\geqs 1/2$ and $\phi(t) = 0$ for all $t\leqs 1/4$.  Let $D^k \subset \bbR^k$ denote the closed unit disk, so $\partial D^k = S^{k-1}$.  
The formula $H^+ (x) = \phi(||x||) f(x/||x||)$ now gives a 
$C^{k+1}$ map $D^k \to Y$ which restricts to $f$ on each shell $\{x\in D^k \,|\,\,\, ||x|| = r\}$ with $r\geqs 1/2$.  Gluing two copies of $H^+$ now gives a $C^{k+1}$ ``null-homotopy'' of $f$ defined on the closed manifold $S^k$.

We now define 
$$A = \{F \in C^{k+1} (S^k, Y)\,\, | \,\,\, F(x) = 0 \textrm{\, for \,} x\in S^{k-1} \subset S^k\}.$$
Note that $A$ is a Banach space:  since $S^k$ is compact,~\cite[Theorem 5.4]{Abraham-Smale} implies that $C^{k+1} (S^k, Y)$ is a Banach space, and $A$ is a closed subspace of 
$C^{k+1}(S^k, Y)$.  (This is the reason for working with $C^{k+1}$ maps rather than smooth ones.)

Next, we define $\rho\co A\to C^{k+1} (S^k, Y)$ by setting $\rho(F) = F+H$.  The evaluation map
$\mathrm{ev}_{\rho}\co A\cross S^k \to Y$ is given by $\mathrm{ev}_{\rho} (F, x) = F(x) + H(x)$.
Since both $(F,x) \mapsto F(x)$ 
and $(F,x) \mapsto x \mapsto H(x)$ are of class $C^{k+1}$, so is their sum (the fact that the evaluation map $(F,x) \mapsto F(x)$ is of class $C^{k+1}$ follows from~\cite[Lemma 11.6]{Abraham-Smale}).

We are now ready to apply the transversality theorem.  Setting $X= S^k$, $W = W_i$ (for some $i\in I$) and with $A$ as above, all the hypotheses of Theorem~\ref{transversality} are clearly satisfied, except for the final requirement that 
$\mathrm{ev}_{\rho} \pitchfork W_i$.  But this is easily seen to be the case.  In fact, the derivative of $\mathrm{ev}_{\rho}$ surjects onto $T_y Y$ for each $y$ in the image of $\mathrm{ev}_{\rho}$, because given a $C^{k+1}$ map $F\co S^k \to Y$ with $F(x) = y$ and a vector $v\in T_y Y$, we may adjust $F$ in a small neighborhood of $x$ so that the map remains $C^{k+1}$ and its derivative hits $v$.  

We now conclude that 
$\{F\in A | \rho_a \pitchfork W_i\}$ is residual in $A$, for each stratum $W_i$.  Since the intersection of countably many residual sets is (by definition) residual, we in fact see that
$$\{F\in A \,\,\,|  \,\,\, \rho_F \pitchfork W_i \,\,\,\forall \, i \in I \}$$ 
is residual, hence dense, in $A$.  
In particular, since $A$ is non-empty, there exists a map
$F\co S^k \to Y$ such that $F|_{S^{k-1}} = f$ and $\rho_F = F+H$ is transverse to each $W_i$.  
Since $k < \mu = \mathrm{codim} (W_i)$, this implies that the image of $F+H$ must be disjoint from each $W_i$.  Hence $(F+H)(S^k) \subset U$, and $f$ is zero in $\pi_{k-1} U$.
\end{proof}

We can now prove the main result of this section.  This result extends work of Ho and Liu, who showed that spaces of flat connections over surfaces are connected~\cite[Theorem 5.4]{Ho-Liu-non-orient}.  We note, though, that their work applies to general structure groups $G$.  We also note that in the orientable case this result is closely related to work of Daskalopoulos and Uhlenbeck~\cite[Corollary 2.4]{Dask-Uhl}, which concerns the less-highly connected space of \emph{stable} bundles.

\begin{proposition}$\label{connectivity}$ Let $M = M^g$ denote a compact Riemann surface of genus $g$, and let $n>1$ be an integer.
Then the space $\flatc^k(n)$ of flat connections on a trivial rank $n$ bundle over $M$ is $2g(n-1)$--connected, and if $\Sigma$ is a non-orientable surface with double cover $M^g$, then the space of flat connections on any principal $U(n)$--bundle over $\Sigma$ is $(g(n-1) -1)$--connected.
\end{proposition}
\begin{proof}  We begin by noting that Sobolev spaces (of sections of fiber bundles) over compact manifolds are always second countable; this follows from Bernstein's proof of the Weierstrass theorem since we may approximate any function by smooth functions, and locally we may approximate smooth functions (uniformly up to the $k$th derivative for any $k$) by Bernstein polynomials.
Since the inclusion $\flatc^k(n)\injects \Css^k(n)$
is a homotopy equivalence (Theorem~\ref{retraction}), the orientable case now follows by applying Corollary~\ref{stratification} (and Lemma~\ref{codim}) to the Harder-Narasimhan stratification.\footnote{Note that $\Css^k(n)$ is in fact open.  This is a slightly subtle point.  Atiyah and Bott provide an ordering on the Harder--Narasimhan strata in which the semi-stable stratum is the minimum stratum, and the closure of any stratum lies in the union of the larger strata.  Hence the complement of $\Css^k(n)$ is the union of the closures of the other strata.  The Harder--Narasimhan stratification is actually locally finite, so this union of closed sets is closed.  Local finiteness can be deduced from convergence of the Yang--Mills flow and the fact that the critical values of the Yang--Mills functional form a discrete subset of $\bbR$ (see R{\aa}de~\cite{Rade}), but also follows from the more elementary methods of Atiyah and Bott, as explained in~\cite{Ramras-YM}.)}

For the non-orientable case, we work in the set-up of non-orientable Yang--Mills theory, as developed by Ho and Liu~\cite{Ho-Liu-non-orient}.
Let $\Sigma$ be a non-orientable surface with double cover $M^g$, and let $P$ be a principal $U(n)$--bundle over $\Sigma$.  Let $\pi\co M^g\to \Sigma$ be the projection, and let $\widetilde{P} = \pi^* P$.  Then the deck transformation $\tau \co M^g\to M^g$ induces an involution $\widetilde{\tau}\co \widetilde{P} \to \widetilde{P}$, and $\widetilde{\tau}$ acts on the space $\A^k(\widetilde{P})$ by pullback.  Connections on $P$ pull back to connections on $\widetilde{P}$, and in fact, the image of the pullback map is precisely the set of fixed points of $\tau$ (see, for example, Ho~\cite{Ho}).  Hence we have a homeomorphism
$\A^k(P) \homeo \A^k(\widetilde{P})^{\widetilde{\tau}}$, which we treat as an identification.  The Yang--Mills functional $L$ is invariant under $\widetilde{\tau}$, and hence its gradient flow restricts to a flow on $\A^k(P)$.  

Assume for the moment that $\fc^k(P) \neq \emptyset$.  
The flat connections on $P$ pull back to flat connections on $\widetilde{P}$, and again the image of $\flatc^k(P)$ in $\A(\widetilde{P})$ is precisely $\flatc^k(\widetilde{P})^{\widetilde{\tau}}$.  If we let 
$\Css^k(P)$ denote the fixed set $\Css^k(\widetilde{P})^{\widetilde{\tau}}$, then the gradient flow of $L$ restricts to give a deformation retraction from $\Css^k(P)$ to $\flatc^k(P)$.  The complement of 
$\Css^k(P)$ in $\A^k(P)$ may be stratified as follows: for each Harder-Narasimhan stratum $\mcC^k_\mu (\wt{P})\subset \A^k(\wt{P}) \isom \mcC^k \left(\wt{P} \cross_{U(n)} \bbC^n\right)$, we consider the fixed set $\mcC^k_{\mu} (P) := \left( \mcC^k_\mu (\wt{P}) \right)^{\wt{\tau}}$.
By Ho and Liu~\cite[Proposition 5.1]{Ho-Liu-non-orient}, $\mcC^k_{\mu} (P)$ is a real submanifold of $\A^k(P)$, and if it is non-empty then its real codimension in $\A^k(P)$ is half the real codimension of $\mcC^k_{\mu} (\wt{P})$ in $\A^k (\wt{P})$.  It now follows from Lemma~\ref{codim} that the codimensions of the non semi-stable strata $\mcC^k_{\mu} (P)$ are at least $g(n-1) + 1$ (this is a rather poor bound; see~\cite{Ramras-YM}).  It now follows from 
Corollary~\ref{stratification} that $\fc^k (P)$ has the desired connectivity.

To complete the proof, we must show that all bundles $P$ over $\Sigma$ actually admit flat connections.  This was originally proven by Ho and Liu~\cite[Theorem 5.2]{Ho-Liu-ctd-comp-II}, and in the current context may be seen as follows.  There are precisely two isomorphism types of principal $U(n)$--bundles over any non-orientable surface.  (A map from $\Sigma$ into $BU(n)$ may be homotoped to a cellular map, and since the 3-skeleton of $BU(n)$ is a 2-sphere, the classification of $U(n)$--bundles is independent of $n$.  Hence it suffices to note that the relative $K$--group $\widetilde{K}_0 (\Sigma)$ has order 2.)  We have just shown that the space of flat connections on each bundle is either empty or connected, so Proposition~\ref{rep-flat} gives a bijection between connected components of $\Hom(\pi_1 \Sigma, U(n))$ and bundles admitting a flat connection.  So it suffices to show that the representation space has at least two components.  This follows easily from the obstruction defined by Ho and Liu~\cite{Ho-Liu-ctd-comp-II}.  
\end{proof}

\begin{remark} In the non-orientable case, some improvement to Theorem~\ref{connectivity} is possible.  The results of Ho and Liu show that many of the Harder-Narasimhan strata for the double cover of non-orientable surface contain no fixed points, and hence the above lower bound on the minimal codimension of the Morse strata is not tight in the non-orientable case.  In the orientable case, the tightness of Lemma~\ref{codim} shows that the bound on connectivity of $\fc^k (n)$ is tight.  This can be proven using the Hurewicz theorem and a homological calculation~\cite{Ramras-YM}.
\end{remark}

As discussed in Section~\ref{gp-comp}, the following results are quite close to the work of Ho and Liu.

\begin{corollary}$\label{rep-ctd}$
For any compact Riemann surface $M$ and any $n\geqs 1$, the representation space
$\Hom(\pi_1 (M), U(n))$ is connected.  In particular, $\Rep(\pi_1 M)$ is stably group-like.
\end{corollary}
\begin{proof}  The genus $0$ case is trivial.  When $n=1$, $U(1) = S^1$ is abelian and all representations factor through the abelianization of $\pi_1 M$.  Hence $\Hom(\pi_1 M, U(1))$ is a product of circles.  When $g, n\geqs 1$ 
we have $2g(n-1) \geqs 0$, so Proposition~\ref{connectivity} implies that 
$\flatc^k(n)$ is connected.  Connectivity of $\Hom(\pi_1(M), U(n))$ follows from Proposition~\ref{rep-flat}, because any $U(n)$ bundle over a Riemann surface which admits a flat connection is isomorphic to $M\cross U(n)$ (for a nice, elementary, and easy proof of this fact, see Thaddeus~\cite[pp. 78--79]{Thaddeus}).
\end{proof}

\begin{corollary}$\label{rep-ctd-no}$
Let $\Sigma$ be a compact, non-orientable, aspherical surface.  Then for any $n\geqs 1$, the representation space $\Hom(\pi_1 \Sigma, U(n))$ has two connected components, and if $\rho\in \Hom(\pi_1 \Sigma, U(n))$ and $\psi\in \Hom(\pi_1 \Sigma, U(m))$ lie in the non-identity components, then $\rho\oplus \psi$ lies in the identity component of $\Hom(\pi_1 \Sigma, U(n+m))$.  In particular, $\Rep(\pi_1 \Sigma)$ is stably group-like.
\end{corollary}
\begin{proof}  First we consider connected components.  The case $n=1$ follows as in Corollary~\ref{rep-ctd}.  When $n>1$, it follows immediately from Proposition~\ref{connectivity} that the space of flat connections on any principal $U(n)$--bundle over $\Sigma$ is connected.

As discussed in the proof of Proposition~\ref{connectivity}, there are precisely two bundles over $\Sigma$, classified by their first Chern classes, and there is a bijection between components of the representation space and isomorphism classes of bundles.  Hence the components of $\Hom(\pi_1 \Sigma, U(n))$ are classified by the Chern classes of their induced bundles, and since Chern classes are additive, the sum of two representations in the non-identity components of $\Hom(\pi_1 \Sigma, U(-))$ lies in the identity component.
\end{proof}


\section{Proof of the main theorem}$\label{main-thm-section}$

We can now prove our analogue of the Atiyah--Segal theorem.

\begin{theorem}$\label{main-thm}$ Let $M$ be a compact, aspherical surface (in other words, $M\neq S^2, \bbR P^2$).  Then for $*>0$,
$$\K_*(\pi_1 (M))\isom K^{-*}(M),$$
where $K^*(M)$ denotes the complex $K$--theory of $M$.  In the non-orientable case, this in fact holds in degree $0$ as well; in the orientable case, we have $\K_0(\pi_1 (M)) \homeo \Z$.
\end{theorem}

We note that the isomorphism in Theorem~\ref{main-thm} is functorial for \emph{smooth} maps between surfaces, as will be apparent from the proof.  In particular, the isomorphism is equivariant with respect to the mapping class group of $M$.
The $K$--theory of surfaces is easily computed (using the Mayer-Vietoris sequence or the Atiyah-Hirzebruch spectral sequence), so Theorem~\ref{main-thm} gives a complete computation of the deformation $K$--groups.

\begin{corollary}$\label{k-groups}$
Let $M^g$ be a compact Riemann surface of genus $g>0$.  Then
 $$\K_*(\pi_1 M^g) = \left\{ \begin{array}{ll}
						\Z, & * = 0\\
						\Z^{2g}, & * \mathrm{\,\,odd}\\
						\Z^2,  & * \mathrm{\,\, even,\,\,} * > 0.
					\end{array}
				\right.
 $$
Let $\Sigma$ be a compact, non-orientable surface of the form $\Sigma = M^g \# N_j$ ($g\geqs 0$), where $j = 1$ or $2$ and $N_1 = \bbR P^2$, $N_2 = \bbR P^2 \# \bbR P^2$ (so $N_2$ is the Klein bottle).  Then if $\Sigma \neq \bbR P^2$, we have:
 $$\K_*(\pi_1 M^g \# N_j) = \left\{ \begin{array}{ll}
						\Z\oplus \Z/2\Z, & *  \mathrm{\,\, even} \\
						\Z^{2g+j-1}, & * \mathrm{\,\,odd.}
					\end{array}
				\right.
 $$
 \end{corollary}

\begin{proof}[Proof of Theorem~\ref{main-thm}]  

{\bf I. The orientable case:}  Let $M = M^g$ be a Riemann surface of genus $g>0$.  
We will exhibit a zig-zag of weak equivalences between $\K(\pi_1 M)$ and $\bbZ\cross \Map^0 (M, BU)),$
where $\Map^0$ denotes the connected component of the constant map. 

By Corollary~\ref{model} and Proposition~\ref{rep-flat} (and the fact that any bundle over a Riemann surface admitting a flat connection is trivial), the zeroth space of the spectrum $\K (\pi_1 M)$ is weakly equivalent to 
\begin{equation}\label{hocolim} \hocolim_{\stackrel{\maps}{\oplus 1}}  \Rep(\pi_1 M)_{hU}
\homeo \hocolim_{\stackrel{\maps}{\oplus [\tau]}} \coprod_{n} EU(n) \cross_{U(n)} \left( \flatc^{k} (n)/\G^{k+1}_0 (n) \right)
\end{equation}
where the maps are induced by direct sum with the trivial connection $\tau$ on the trivial line bundle.  Since the based gauge groups $\G^{k+1}_0 (n)$ act freely on $\A^{k}(n)$, and (by Mitter--Viallet~\cite{Mitter-Viallet} or Fine--Kirk--Klassen~\cite{FKK}) the projection maps are locally trivial principal $\G^{k+1}_0 (n)$--bundles, a basic result about homotopy orbit spaces~\cite[13.1]{A-B} shows that we have a weak equivalence
\begin{equation}\label{free-action} E\G^{k+1}(n) \cross_{\G^{k+1}(n)} \flatc^{k} (n) \stackrel{\heq}{\maps} EU(n) \cross_{U(n)} \left( \flatc^{k} (n)/\G^k_0 (n) \right).
\end{equation}
It now follows from (\ref{free-action}) that the mapping telescope (\ref{hocolim}) is weakly equivalent to
$$\hocolim_{\stackrel{\maps}{\oplus \tau}} \coprod_{n} \left(\flatc^k(n)\right)_{h\G^{k+1}(n)}
\heq \Z\cross \hocolim_{\stackrel{\maps}{\oplus \tau}} \flatc^k(n)_{h\G^{k+1} (n)}.$$
Proposition~\ref{connectivity} shows that the connectivity of the projections 
$\flatc^k(n)_{h\G^{k+1}(n)} \to B\G^{k+1}(n)$ tends to infinity, and since the homotopy groups of a mapping telescope may be described as colimits, these maps induce a weak equivalence
\begin{equation}\label{css-c}
\Z\cross \hocolim_{n\to\infty} \flatc^k(n)_{h\G^{k+1}(n)} 
	\maps \Z\cross \hocolim_{n\to \infty} B\G^{k+1}(n).
\end{equation}
By Lemma~\ref{smoothing}, the inclusion $\G^{k+1}(n)\injects \G(n)$ is a weak equivalence, so we may replace $\G^{k+1}(n)$ with $\G(n)$ on the right.

We have been using Milnor's functorial model $E^J(-)\to B^J(-)$ for universal bundles (see Remark~\ref{mixed}).  Atiyah and Bott have shown~\cite[Section 2]{A-B} that the natural map
$$\Map (M, EU(n)) \to \Map^0 (M, BU(n))$$
is a universal principal $\Map(M, U(n)) = \G(n)$ bundle, where again $\Map^0$ denotes the connected component of the constant map.  As in Remark~\ref{mixed}, the ``mixed model" for $B\G(n)$ gives a zig-zag of weak equivalences
$$B\G(n) \longleftarrow \left(E\G(n) \cross \Map (M, EU(n)) \right)/\Map(M, U(n)) \maps \Map^0(M, BU(n)),$$
which are natural in $n$ and hence induce weak equivalences on mapping telescopes (formed using the standard inclusions $U(n)\injects U(n+1)$).  
The projection 
$$\hocolim_{n\to \infty} \Map^0 (M, BU(n))\maps \colim_{n\to \infty} \Map^0 (M, BU(n)) 
=\Map^0 (M, BU)$$
is a weak equivalence, since maps from compact sets into a colimit land in some finite piece.  
This completes the desired zig-zag.

\vspace{.1in}

\noindent {\bf II. The non-orientable case:} Let $\Sigma\neq \bbR P^2$ be a non-orientable surface.  Once again, Proposition~\ref{model} and Proposition~\ref{rep-flat} tell us that the zeroth space of $\K(\pi_1 \Sigma)$ is weakly equivalent to 
$$\hocolim_{\stackrel{\maps}{\oplus \tau}} \coprod_{[P_i]} \left(\flatc^k(P_i)\right)_{h\G^{k+1}(P_i)},$$
where the disjoint union is taken over all $n$ and over all isomorphism types of principal 
$U(n)$--bundles.
By Proposition~\ref{connectivity} we know that $\flatc^k(P_i)$ is 
$(g(\widetilde{\Sigma})(n_i-1) - 1)$--connected, where $n_i = \mathrm{dim} (P_i)$ and
$g(\widetilde{\Sigma})$ denotes the genus of the orientable double cover of $\Sigma$.  Since we have assumed $\Sigma\neq \bbR P^2$, we know that $g(\widetilde{\Sigma})>0$, and hence the connectivity of $\flatc^k(P_i)$ tends to infinity with $n_i$.  Hence the natural map
\begin{equation}\label{forget}
\hocolim_{\stackrel{\maps}{\oplus \tau}} \coprod_{[P_i]} \left(\flatc^k(P_i)\right)_{h\G^{k+1}(P_i)}
	\maps \hocolim_{\stackrel{\maps}{\oplus 1}} \coprod_{[P_i]} B\G^{k+1}(P_i)
\end{equation}
is a weak equivalence (on the right hand side, $1$ denotes the identity element in $\G^{k+1}(1)$). 
As in the orientable case, we may now switch to the Atiyah-Bott models for
$B\G(P_i)$, obtaining the space
$$\hocolim_{\stackrel{\maps}{\oplus 1}} \coprod_{[P_i]} \Map^{P_i} (\Sigma, BU(n_i)),$$
where $\Map^{P_i}$ denotes the component of the mapping space consisting of those maps
$f\co\Sigma\to BU(n_i)$ with $f^*(EU(n_i))$ isomorphic to $P_i$.  But since the union is taken over \emph{all} isomorphism classes, this space is homotopy equivalent to 
$$\Z \cross \hocolim_{n\to \infty} \Map (\Sigma, BU(n)) \heq \Z\cross \Map(\Sigma, BU) = \Map(\Sigma, \Z\cross BU).$$
\end{proof}

We briefly discuss the spectrum-level version of Theorem~\ref{main-thm}.
The space level constructions used in the proof of Theorem~\ref{main-thm} can be lifted to spectrum-level constructions.  This involves constructing a variety of spectra (and maps between them) including, for instance, a spectrum arising from a topological category of flat connections and gauge transformations.  Each spectrum involved can be constructed from a $\Gamma$--space in the sense of Segal~\cite{Segal-cat-coh}, and the space-level constructions above essentially become weak equivalences between the group completions of the monoids underlying these $\Gamma$--spaces.  Since these group completions are weakly equivalent to the zeroth spaces of these $\Omega$-spectra, the space-level weak equivalences lift to weak-equivalences of spectra.  In the non-orientable case, one concludes that $\K(\pi_1 M)$ is weakly equivalent to the function spectrum $F(M, \mathbf{ku})$.  The end result in the orientable case is somewhat uglier, due to the failure of Theorem~\ref{main-thm} on $\pi_0$.  In this case, $\K(\pi_1 M)$ is weakly equivalent to a subspectrum of $F(M, \mathbf{ku})$, essentially consisting of those maps homotopic to a constant map.  
One may ask whether the intermediate spectra are $\mathbf{ku}$--algebra spectra and whether the maps between them preserve this structure.  More basically, one may ask whether the isomorphisms in Theorem~\ref{main-thm} come from a homomorphism of graded rings.  Recall that T. Lawson~\cite{Lawson-prod} has constructed a $\mathbf{ku}$--algebra structure on the spectrum $\K(G)$.  Constructing a compatible ring structure for the spectrum arising from flat connections appears to be a subtle problem.  This problem, and the full details of the spectrum-level constructions, will be considered elsewhere.

We now make the following conjecture regarding the homotopy type of the spectrum $\K (\pi_1 M)$, as a algebra over the connective $K$--theory spectrum $\mathbf{ku}$.
Note that it is easy to check (using Theorem~\ref{main-thm}) that the homotopy groups of the proposed spectrum are the same as 
$\K_*(\pi_1 (M))$.

\begin{conjecture}$\label{main-conjecture}$
For any Riemann surface $M^g$, the spectrum $\K (\pi_1 M)$ is weakly equivalent, as a $\mathbf{ku}$--algebra, to $ \mathbf{ku}\vee \left( \bigvee_{2g} \Sigma \mathbf{ku} \right) \vee \Sigma^2 \mathbf{ku}$.
\end{conjecture}


\section{The stable moduli space of flat connections}$\label{coarse-moduli}$

In this section we study the coarse moduli space of flat unitary connections over a surface, after stabilizing with respect to rank.  By definition, the moduli space of flat connections over a compact manifold, with structure group $G$, is the space
$$\coprod_{[P_i]}\fc^{k,p} (P)/\G^{k+1, p} \isom \Hom(\pi_1 M, G)/G,$$
where the disjoint union is taken over isomorphism classes of principal $G$--bundles.
This homeomorphism follows immediately from Proposition~\ref{rep-flat} (so long as $k$, $p$, and $G$ satisfy the hypotheses of that result).  In particular, the moduli space of flat unitary connections is simply $\Hom(\pi_1 M, U(n))/U(n)$, and the inclusions $U(n)\injects U(n+1)$ allow us to stabilize with respect to rank.  The colimit $\mcM(\pi_1 M) $ of these spaces is just $\Hom(\pi_1 M, U)/U$, where $U = \colim U(n)$ is the infinite unitary group.  We call this space the stable moduli space of flat (unitary) connections.  (This rather simple stabilization suffices for surface groups, but in general must be replaced by a subtler construction; see Remark~\ref{stabilization}.)

T. Lawson~\cite{Lawson-simul} has exhibited a surprising connection between deformation $K$--theory and this stable moduli space.  His results suggest that for many compact, aspherical manifolds, only finitely many homotopy groups of this stable moduli space are non-zero, and in fact each component of this space should have the homotopy type of a finite product of Eilenberg-MacLane spaces (see Corollary~\ref{EM}).  This situation is closely tied to Atiyah--Segal phenomena in deformation $K$--theory: we expect that the homotopy groups of $\K(M)$ will agree with $K^{-*} (M)$ for $*$ at least one less than the cohomological dimension of the group, and after this point the homotopy groups of the stable moduli space should vanish.  We now explain Lawson's results and how they play out for surface groups.

For the remainder of this section, we think of $\K(\Gamma)$ as the connective spectrum described in Section~\ref{K-def}.  
Lawson's theorem states that for any finitely generated group $\Gamma$, there is a homotopy cofiber sequence of spectra
\begin{equation}\label{cofib}
\Sigma^2 \K(\Gamma) \stackrel{\beta}{\maps} \K(\Gamma) \maps \Rdef (\Gamma),
\end{equation}
and a corresponding long exact sequence in homotopy 
\begin{equation}\label{Bott-LES}
\cdots \maps\K_{*-2}(\Gamma)  \stackrel{\beta_*}{\maps}  \K_*(\Gamma) 
\maps \Rdef_* (\Gamma) \maps \cdots.
\end{equation}
Here $\Sigma^2 \K(\Gamma)$ denotes the second suspension of $\K(\Gamma)$ and $\Rdef (\Gamma)$ denotes the ``deformation representation ring" of $\Gamma$, defined below.  Note that for any spectrum $X$, one has $\pi_*(\Sigma X) = \pi_{*-1} X$, hence the degree shift in the long exact sequence (\ref{Bott-LES}).

As we will explain, the cofiber $\Rdef(\Gamma)$ is quite closely linked to the stable moduli space $\mcM(\Gamma)$.  The first map $\beta$ in the cofibration sequence (\ref{cofib}) is the Bott map in deformation $K$--theory, and is obtained from the Bott map in connective $K$--theory $\mathbf{ku}$ by smashing with $\K(\Gamma)$; this requires the $\mathbf{ku}$--module structure in deformation $K$--theory constructed by Lawson~\cite{Lawson-prod}.  

Lawson's construction of the Bott map relies on the modern theory of structured ring spectra.  In particular, his results require the model categories of module and algebra spectra studied by Elmendorf, Kriz, Mandell, and May~\cite{EKMM}, Elmendorf and Mandell~\cite{EM}, and Hovey, Shipley, and Smith~\cite{HSS}.

His work makes rigorous the following purely heuristic construction (which we include simply to provide some intuition).
We may consider the Bott element $\beta\in \pi_2 BU(n) = \bbZ$ as a family of representations via the map $BU(n) \maps EU(n)\cross_{U(n)} \Hom(\Gamma, U(n))$ given by $x\mapsto [\tilde{x}, I_n]$ where $\tilde{x}\in EU(n)$ is any lift of $x$; this is well-defined since $I_n$ is fixed under conjugation.  If 
$\Gamma$ is stably group-like, Theorem~\ref{gp-comp} allows us to think of homotopy classes in $\K_m(\Gamma)$
as families of representations $\rho\co S^{m} \to EU(n)\cross_{U(n)} \Hom(\Gamma, U(n))$.  Now tensoring with $\beta$ gives a new family $\rho\otimes \beta\co S^{m-2}\wedge S^2 = S^m\to \Hom(\Gamma, U(n))_{hU(n)}$, via the formula
$\rho\otimes\beta (z\smash w) = \rho(z)\otimes \beta(w)$.  Of course some care needs to be taken in defining this tensor product, since $\rho(z)$ and $\beta(w)$ lie in the \emph{homotopy orbit spaces}, rather than simply in the representation spaces (and we have also ignored questions of basepoints and well-definedness.)  Lawson's construction of the Bott map~\cite{Lawson-prod} is in practice quite different from this hands-on approach, and it would interesting to have a rigorous proof that the two agree.

Since $\K(\Gamma)$ is connective, $\pi_0 \Sigma^2 \K(\Gamma)$ 
and $\pi_1  \Sigma^2 \K(\Gamma)$ 
are zero, and hence the long exact sequence (\ref{Bott-LES}) immediately gives isomorphisms
\begin{equation}\label{pi_1}\K_i (\Gamma)  \isom \pi_i \Rdef (\Gamma) \end{equation}
for $i=0, \, 1$, as well as an exact sequence
\begin{equation}\label{pi_2}
\K_0(\Gamma) \stackrel{\beta_*}{\maps} \K_2(\Gamma) \maps \Rdef_2 \maps 0
\end{equation}
(is is not known whether the first map is injective in general).

If $\K_*(\Gamma)$ agrees with the periodic cohomology theory $K^{-*} (B\Gamma)$ for large $*$, then one should expect the Bott map to be an isomorphism after this point (certainly this would follow from a sufficiently natural correspondence between deformation $K$--theory of $\Gamma$ and topological $K$--theory of $B\Gamma$).  From the long exact sequence (\ref{Bott-LES}), one would then conclude, as mentioned above, that $\pi_* (\Rdef (\Gamma))$ vanishes in high degrees.

\begin{remark} In Section~\ref{K-def}, we described $\K(\Gamma)$ as the connective spectrum associated to a permutative category of representations, and we computed its homotopy groups in Theorem~\ref{main-thm}. Lawson works with a different model, built from the $H$--space $\coprod_n V(n) \cross_{U(n)} \Hom(\Gamma, U(n))$~\cite{Lawson-prod}.  Here $V(n)$ denotes the infinite Stiefel manifold of $n$--frames in $\bbC^{\infty}$.  One may interpolate between the two models by using a spectrum built from the $H$--space 
$$(EU(n) \cross V(n)) \cross_{U(n)} \Hom(\Gamma, U(n)),$$
and hence Theorem~\ref{main-thm} computes the homotopy groups of Lawson's deformation $K$--theory spectrum as well.
\end{remark}

We now describe the deformation representation ring and its relation to the stable moduli space.
Given any topological abelian monoid $A$ (for which the inclusion of the identity is a cofibration), one may apply Segal's infinite loop space machine~\cite{Segal-cat-coh} to produce a connective $\Omega$--spectrum; equivalently the bar construction $BA$ (the realization of the simplicial space $[n]\mapsto A^n$, with face maps given by multiplication and degeneracies given by insertion of the identity~\cite{Segal-class-ss}) is again an abelian topological monoid and one may iterate.  In particular, the zeroth space of this spectrum is $\Omega BA$.  The deformation representation ring $\Rdef (\Gamma)$ is the spectrum associated to the abelian topological monoid
$$\overline{\mathrm{Rep} (\Gamma)} = \coprodmo_{n=0}^{\infty} \Hom(\Gamma, U(n))/U(n),$$
so we have
$$\pi_* \Rdef (\Gamma) \isom \pi_* \Omega B \left( \overline{\mathrm{Rep} (\Gamma)} \right).$$
It is in general rather easy to identify the group completion $\Omega BA$ when $A$ is an abelian monoid.  

\begin{proposition}$\label{rep-ring}$
Let $\Gamma$ be a finitely generated discrete group, and assume that $\Rep(\Gamma)$ is stably group-like with respect to the trivial representation $1\in \Hom(\Gamma, U(1))$ (e.g. $\Gamma = \pi_1 M$ with $M$ a compact, aspherical surface).  Then the zeroth space of $\Rdef (\Gamma)$ is weakly equivalent to $\bbZ \cross \Hom(\Gamma, U)/U$.
Hence for $*>0$ we have 
$$\pi_* \Hom(\Gamma, U)/U \isom \pi_* \Rdef (\Gamma),$$
and it follows from (\ref{pi_1}) that $\pi_1 \Hom(\Gamma, U)/U \isom \K_1(\Gamma)$.
\end{proposition}
\begin{proof}  If $\Rep(\Gamma)$ is stably group-like with respect to the trivial representation $1\in \Hom(\Gamma, U(1))$, then the same is true for the monoid of isomorphism classes $\overline{\mathrm{Rep} (\Gamma)}$.
As in Section~\ref{gp-comp}, we can now apply Ramras~\cite[Theorem 3.6]{Ramras-excision}.  (That result has one additional hypothesis -- the representation $1$ must be ``anchored" -- but this is trivially satisfied for abelian monoids.)  Hence
$$\Omega B \left(\overline{\mathrm{Rep} (\Gamma)} \right) \heq \hocolim_{\stackrel{\maps}{\oplus 1}} \overline{\mathrm{Rep} (\Gamma)} \heq \bbZ \cross \hocolim_{n\to \infty} \Hom(\Gamma, U(n))/U(n),$$
and to complete the proof it suffices to check that the projection
$$\hocolim_{n\to\infty} \Hom (\Gamma, U(n))/U(n)
\maps \colim_{n\to \infty} \Hom (\Gamma, U(n))/U(n)$$
is a weak equivalence.
But this follows from the fact that in both spaces, compact sets land in some finite piece (for the colimit, this requires that points are closed in $\Hom(\Gamma, U(n))/U(n)$; this space is in fact Hausdorff because the orbits of $U(n)$ are compact, hence closed in $\Hom(\Gamma, U(n))$, which is a metric space, hence normal).
\end{proof}

\begin{remark}$\label{stabilization}$
When $\Rep(\Gamma)$ is not stably group-like with respect to the trivial representation, a more complicated stabilization process can be used to obtain a concrete model for the zeroth space of $\Rdef (\Gamma)$.  Of course if $\Rep(\Gamma)$ is stably group-like with respect to some other representation $\rho$, we can simply replace block sum with $1$ by block sum with $\rho$.  If there is no such representation $\rho$, then we proceed by means of a rank filtration: the submonoids 
$\overline{\mathrm{Rep}_n (\Gamma)} \subset \overline{\mathrm{Rep} (\Gamma)}$ generated by representations of dimension at most $n$ are finitely generated (by any set of representatives for the finite sets $\pi_0 \Hom(G, U(n))$) hence stably group-like with respect to the sum $\Phi_n$ of all the generators.  One now obtains a weak equivalence between the zeroth space of $\Rdef (\Gamma)$ and the colimit
$\colim_{n\to \infty} (\colim_{\stackrel{\maps}{\oplus \Phi_n}} \overline{\mathrm{Rep}_n (\Gamma)})$.
The proof is similar to the arguments in~\cite[Section 5]{Ramras-excision}, and will not be needed here.
\end{remark}

We can now show, as promised above, that when $\Rep(\Gamma)$ is stably group-like, each component of the stable moduli space of flat connections has the homotopy type of a product of Eilenberg-MacLane spaces.

\begin{corollary}$\label{EM}$ Let $\Gamma$ be a finitely generated discrete group, and assume that $\Rep(\Gamma)$ stably group-like with respect to the trivial representation.  Then the stable moduli space $\mathcal{M}(\Gamma) = \Hom(\Gamma, U)/U$ is homotopy equivalent to 
$$\pi_0 \mathcal{M}(\Gamma) \cross \prod_{i=0}^\infty K(\pi_i \mathcal{M} (\Gamma), i)$$
where $K(\pi, i)$ denotes an Eilenberg-MacLane space.
\end{corollary}
\begin{proof} By Proposition~\ref{rep-ring}, each component of $\Hom(\Gamma, U)/U$ is homotopy equivalent to a path component of the zeroth space of $\Rdef (\Gamma)$.  
As discussed above, this zeroth space is the loop space on the abelian topological monoid 
$B(\overline{\mathrm{Rep} (\Gamma)})$.  Recall that any connected abelian topological monoid is weakly equivalent to a product of Eilenberg-MacLane spaces (Hatcher~\cite[Corollary 4K.7]{Hatcher}).  Using the abelian monoid structure on the loop space derived from point-wise multiplication of loops (rather than concatenation) one now sees that the identity component of $\Omega B(\overline{\mathrm{Rep} (\Gamma)}$ is a product of Eilenberg-MacLane spaces.
But  $\Omega B(\overline{\mathrm{Rep} (\Gamma)}$ is a group-like $H$--space, so each of its path components is homotopy equivalent to the identity component, completing the proof.  (Note that each space in question has the homotopy type of a CW-complex; $\Hom(\Gamma, U(n))/U(n)$ is a CW-complex by Park and Suh~\cite{Park-Suh}.)
\end{proof}

Combining Proposition~\ref{rep-ring} with (\ref{pi_1}) and Theorem~\ref{main-thm} yields:

\begin{corollary}$\label{moduli}$
For any compact, aspherical surface $M$, the fundamental group of the stable moduli space of flat unitary connections on $M$ is isomorphic to the complex $K$--group $K^{-1} (M)$.  Equivalently, if $M^g$ is a Riemann surface of genus $g$,
$$\pi_1 \left( \Hom(\pi_1 M^g, U)/U \right) \isom \Z^{2g},$$
and in the non-orientable cases (letting $K$ denote the Klein bottle) we have 
$$\pi_1 \left( \Hom(\pi_1 M^g\# \bbR P^2, U)/U \right) \isom \Z^{2g} \,\,\,\,
\mathrm{and} \,\,\,\,
\pi_1 \left( \Hom(\pi_1 M^g\# K, U)/U \right) \isom \Z^{2g+1}.$$
\end{corollary}

When $M$ is orientable we can also calculate the second homotopy group of the stable moduli space.

\begin{proposition}$\label{moduli2}$
Let $M^g$ be a Riemann surface of genus $g\geqs 1$.  Then 
$$\pi_2 \left( \Hom(\pi_1 M^g, U)\right)/U \isom \bbZ.$$
\end{proposition}
\begin{proof}  In light of the exact sequence (\ref{pi_2}) and Proposition~\ref{rep-ring}, it suffices to compute the cokernel of the Bott map $\K_0(\pi_1 M^g)\to \K_2(\pi_1 M^g)$.  Recall that in~\cite{Lawson-prod}, the Bott map arises as multiplication by the element $q^*(b)$, where $q\co \Gamma \to \{1\}$ is the projection and $b\in \pi_2 \K\{1\} \isom \pi_2 \mathbf{ku}$ is the canonical generator (we will see below that the computation of $\K\{1\}$ follows from Corollary~\ref{model}).  Since $\K_0(\pi_1 M^g) \isom \bbZ$ is generated by the unit of the ring $\K_*(\pi_1 M^g)$, it follows that $\beta(\K_0(\pi_1 M^g))\subset \K_2(\pi_1 M^g)$ is generated by $\beta([1]) = q^*(b) \cdot [1] = q^* (b)$.  Hence we simply need to understand the image of the canonical generator $b\in \pi_2 \mathbf{ku}$ under the map $q^*$.
The inclusion $\eta\co \{1\}\injects \pi_1 M^g$ induces a splitting of $q^*$, and since $\K_2(\pi_1 M^g) \isom \Z\oplus \Z$ and $\pi_2 \mathbf{ku} \isom \bbZ$, 
the proposition now follows from the elementary fact that in any diagram 
$$\xymatrix{
	\Z\oplus \Z \ar[r]_p & \Z \ar@/_.5pc/[l]_s
		}
$$
with $p\circ s = \mathrm{Id}$, the cokernel of $s$ is $\bbZ$ (briefly, if $s(1) = (s_1, s_2)$, then
$1 = p(s_1, s_2) = p(1, 0) s_1 + p(0,1) s_2$ so $s_1$ and $s_2$ are relatively prime, and it follows that $\mathrm{coker} (s)$ is torsion-free).
\end{proof}

In the non-orientable case, an analogous argument shows that the cokernel of the map $\K_0(\pi_1 \Sigma)\stackrel{\beta}{\to} \K_2(\pi_1 \Sigma)$ is either $0$ or $\bbZ/2$ (recall that both of these groups are isomorphic to $\bbZ\oplus \bbZ/2$), and hence by (\ref{pi_2}), $\Rdef_2 (\pi_1 \Sigma)$ is either $0$ or $\bbZ/2$. We expect that $\beta$ is in fact an isomorphism in degree zero, and hence that $\Rdef_2 (\pi_1 \Sigma)$ is zero.

For orientable surfaces $M^g$ ($g>0$), it would follow from Conjecture~\ref{main-conjecture} that the Bott map is an isomorphism above degree zero, and consequently $\pi_* \mcM( \pi_1 M^g) = 0$ for $*\geqs 3$.  Our computation of $\pi_i \mcM (\pi_1 M^g)$  ($i=0, 1, 2$), together with Corollary~\ref{EM}, would then imply that $\mcM (\pi_1 M^g)$ is homotopy equivalent to the infinite symmetric product $\mathrm{Sym}^{\infty} (M^g)$.
In the case $g=1$, this is a well-known fact, and follows easily from simultaneous diagonalizability of commuting matrices.  

In the non-orientable case the situation appears to be somewhat different.  There, the isomorphism with complex $K$--theory begins in dimension zero, and hence we expect that the Bott map is always an isomorphism.  Hence we expect that the homotopy groups of $\Hom(\pi_1 \Sigma, U)/U$ vanish above dimension 1, i.e. this space has the homotopy type of a product of circles.  The precise meaning of the homotopy groups $\pi_* \Hom(\Gamma, U)/U$ thus seems rather mysterious.
The reader should note the similarity between these calculations and the main result of Lawson's paper~\cite{Lawson-simul}, which states that $U^k/U$, the space of isomorphism classes of representations of a \emph{free} group, has the homotopy type of $\mathrm{Sym}^{\infty} (S^1)^k = \mathrm{Sym}^{\infty} B(F_k)$.   (Of course this space is homotopy equivalent to $(S^1)^k$.)

\section{Connected sum decompositions}$\label{excision}$

In this section we consider the behavior of deformation $K$--theory on connected sum decompositions of Riemann surfaces.  
Given an amalgamation diagram of groups, applying deformation $K$--theory results in a pull-back diagram of spectra.  An excision theorem states that the natural map
$$\Phi\co \K( G*_K H) \maps \mathrm{holim} \left( \K (G) \maps \K (K) \longleftarrow  \K (H) \right)$$
is an isomorphism, where holim denotes the homotopy pullback.  

Associated to a homotopy cartesian diagram of spaces
\begin{equation}\label{cart}
\xymatrix{
       {W} \ar[r]^{f} \ar[d]^{g}
       & {X} \ar[d]^{h} \\           
       {Y} \ar[r]^{k} 
       & {Z}  
                     }
\end{equation}
there is a long exact ``Mayer--Vietoris'' sequence of homotopy groups
\begin{equation}\label{LES}
 \ldots \maps \pi_k (W)     \stackrel{f_* \oplus g_*}{\maps} 
     \pi_k (X)\oplus \pi_k (Y)  \stackrel{h_* - k_*}{\maps} 
     \pi_k (Z) \stackrel{\partial}{\maps} 
     \pi_{k-1} (W)\maps \ldots
\end{equation}
(this follows from Hatcher~\cite[p. 159]{Hatcher}, together with the fact that the homotopy fibers of the vertical maps in a homotopy cartesian square are weakly equivalent).  If the diagram (\ref{cart}) is a diagram in the category of group-like $H$--spaces (i.e. $H$--spaces for which $\pi_0$ is a group), then all of the maps in the sequence (\ref{LES}) (including the boundary map) are homomorphisms in dimension zero.  Hence whenever deformation $K$--theory is excisive on an amalgamation diagram, one obtains a long-exact sequence in $\K_*$.  

Deformation $K$--theory can fail to satisfy excision in low dimensions, and in particular the failure of Theorem~\ref{main-thm} in degree zero leads to a failure of excision for connected sum decompositions of Riemann surfaces.  We briefly describe this situation.

Letting $M= M^{g_1+g_2}$ denote the surface of  genus $g_1 + g_2$ and $F_k$ the free group on $k$ generators, if we think of $M$ as a connected sum then the Van Kampen Theorem gives us an amalgamation diagram for $\pi_1 M$.  The long exact sequence coming from excision would end with
\begin{equation*}
\begin{split}
\K_1 (F_{2g_1}) \oplus \K_1 (F_{2g_2}) & \maps  \K_1 (\bbZ)  \maps  \K_0(\pi_1 M) \\
& \maps  \K_0(F_{2g_1}) \oplus \K_0(F_{2g_2})  \surjects  \K_0(\bbZ).
\end{split}
\end{equation*}
The groups in this sequence are known, and so the sequence would have the form 
$$\K_1 (F_{2g_1}) \oplus \K_1 (F_{2g_2}) \maps \K_1 (\bbZ) = \bbZ
         \maps \bbZ \maps \bbZ \oplus \bbZ \surjects \bbZ.
$$
We claim, however, that the maps $\K_1(F_{2g_i}) \to \K_1(\bbZ)$ are zero.  This leads immediately to a contradiction, meaning that no such exact sequence can exist and excision is not satisfied in degree zero.

If we write the generators of $F_{2g_i}$ as $a^i_1, b^i_1, \ldots, a^i_{g_i}, b^i_{g_i}$, then the map $c_i\co\bbZ \to F_{2g_i}$ is the multiple-commutator map, sending $1\in \Z$ to $\prod_{j=1}^{g_i} [a^i_j, b^i_j]$.  Since the representation spaces of $F_k$ are always connected,  $\Rep(F_k)$ is stably group-like with respect to $1\in \Hom(F_k, U(1))$.  Hence (using Theorem~\ref{gp-completion-cor}) one finds that the induced map $c_i^*\co\K_*(F_{2g_i}) \to \K_* (\Z)$ may be identified with the map
$$\pi_* (\Z\cross (U^{2g_i})_{hU}) \to \pi_* (\Z\cross U_{hU})$$
induced by the multiple commutator map $C\co U^{2g_i} \to U$ (here the actions of $U$ are via conjugation).  The induced map $C_*$ on homotopy is always zero, and from the diagram of fibrations
$$\xymatrix{
	{U^{2g_i}} \ar[r] \ar[d]^C 
	& {EU\cross_{U} U^{2g_i}} \ar[r] \ar[d] 
	& {BU} \ar[d]^(.45){\begin{turn}{270} $=$ \end{turn} } \\
	{U} \ar[r] 
	& {EU\cross_U U} \ar[r] 
	& {BU}
		  }
$$
one now concludes (using Bott Periodicity) that $c_i^*$ is zero for $*$ odd. 
This shows that deformation $K$--theory is not excisive on $\pi_0$ for connected sum decompositions.  However, based on Theorem~\ref{main-thm} we expect that excision will hold in all higher degrees.


\section{Appendix: Holonomy of flat connections}$\label{holonomy}$

We now discuss the holonomy representation associated to a flat connection on a principal $G$--bundle ($G$ a Lie group) over a smooth, connected manifold $M$.  We show that holonomy induces a bijection from the set of all such (smooth) connections to $\Hom(\pi_1 M, G)$, after taking the action of the based gauge group into account (Proposition~\ref{hol}).  This is well-known, but there does not appear to be a complete reference.  Some of the results to follow appear in Morita's books~\cite{Morita-gdf, Morita-gcc}, and a close relative of the main result is stated in the introduction to Fine--Kirk--Klassen~\cite{FKK}.  

Most proofs are left to the reader; these are generally tedious but straightforward unwindings of the definitions.  Usually a good picture contains the necessary ideas.  Many choices must be made in the subsequent discussion, starting with a choice of left versus right principal bundles.  It is quite easy to make incompatible choices, especially because these may cancel out later in the argument.  We have carefully made consistent and correct choices.

Our principal bundles will always be equipped with a
$\emph{right}$ action of the structure group $G$.
A connection on $P$ is a $G$--equivariant splitting of the natural map $TP\to \pi^*TM$.  The gauge group $\G(P)$ is the group of all equivariant maps $P\stackrel{\phi}{\maps} P$ such that $\pi \circ \phi = \pi$; the gauge group acts on the left of $\A(P)$ via pushforward: $\phi_* A = D\phi \circ A \circ \tilde{\phi}^{-1}$.

Given a smooth curve $\gamma\co [0,1] \to M$ we may define a parallel transport operator 
$T_{\gamma} \co P_{\gamma (0)} \to P_{\gamma(1)}$ by following $A$--horizontal lifts of the path $\gamma$.  An $A$--horizontal lift of $\gamma$ is a curve $\tilde{\gamma}\co [0,1] \to P$ satisfying
$$\pi \circ \tilde{\gamma} = \gamma \quad \text{and} \quad \tilde{\gamma}_p'(t)= A\left(\gamma'(t), \tilde{\gamma}_p(t)\right),$$
and is uniquely determined by its starting point $\tilde{\gamma} (0)$; we denote the lift starting at $p\in P_{\gamma(0)}$ by $\tilde{\gamma}_p$.
Parallel transport is now defined by $T_\gamma(p) = \tilde{\gamma}_p(1).$

Parallel transport is $G$--equivariant and behaves appropriately with respect to composition and reversal of paths.
Any \emph{flat} connection $A$ is locally trivial (see Donaldson--Kronheimer~\cite[Theorem 2.2.1]{DK}), and a standard compactness argument shows that parallel transport is homotopy invariant for such connections.  

\begin{definition} Let $P$ be a principal $G$--bundle over $M$, and choose basepoints $m_0 \in M$, $p_0 \in P_{m_0}$.  Associated to any flat connection $A$ on $P$, the \emph{holonomy representation} $$\rho_A\co \pi_1 (M, m_0) \to G$$ is defined by setting $\rho_A([\gamma])$ to be the unique element of $G$ satisfying $p_0 = T_\gamma^A(p_0)\cdot\rho_A([\gamma])$.
Here $\gamma\co I \to M$ is a smooth loop based at $m_0$ and $[\gamma]$ is its class in $\pi_1 (M, m_0)$.
\end{definition}

We now assume that $M$ is equipped with a basepoint $m_0 \in M$, and we equip all principal bundles $P$ with basepoints $p_0 \in P_{m_0}$.  
We denote the set of all (smooth) flat connections on a principal bundle $P$ by $\fc(P)$. 

\begin{proposition}\label{gauge}
For any $A \in \fc(P)$ and any $\phi \in \G(P)$ we have
$$\rho_{\phi_\ast A} = \phi_{m_0}\rho_A\phi_{m_0}^{-1},$$ where $\phi_{m_0} \in G$ is the unique element such that $p_0 \cdot \phi_{m_0} = \phi(p_0).$  (Note that $\phi \mapsto \phi_{m_0}$ is a homomorphism $\G(P) \to G$.)
\end{proposition}

Proposition~\ref{gauge} shows that we have a diagram
$$\xymatrix{
	\coprod \limits_{[P]} \fc (P) \ar[dr] \ar[rr]^{\mcH} & & \Hom(\pi_1 (M, m_0), G) \\
	& \coprod \limits_{[P]} \fc (P)/\bG (P) \ar[ur]^{\overline{\mcH}},
		}
$$
where $\mathcal{H}(A) = \rho_A$.  The disjoint unions range over some chosen set of representatives for the \emph{unbased} isomorphism classes of (based) principal $G$--bundles.  (In other words, we choose a set of representatives for the unbased isomorphism classes, and then choose, arbitrarily, a basepoint in each representative, at which we compute holonomy.)
We now explain the equivariance properties of this diagram.  When $G$ is connected, Lemma~\ref{restriction} shows that $G$ acts on the space
$\coprod_{[P]} \fc(P)/\bG(P)$.
The action of $g \in G$ on an equivalence class $[A] \in \fc(P) / \bG(P)$ is given by
$g\cdot[A] = [(\phi^g)_\ast A]$,
where $\phi^g \in \G(P)$ is any gauge transformation satisfying $(\phi^g)_{m_0} = g$.  
We can now state the main result of this appendix.

\begin{proposition}\label{hol} The holonomy map defines a (continuous) bijection
$$\overline{\mathcal{H}}\co \coprod_{[P]} \fc(P)/\bG(P) \to \Hom(\pi_1M, G).$$ 
If $G$ is connected, this map is $G$--equivariant with respect to the above $G$--action.
\end{proposition}
\begin{proof}  We begin by noting that equivariance is immediate from Proposition \ref{gauge}, and continuity of the holonomy map is immediate from its definition in terms of integral curves of vector fields (here we are thinking of the $C^\infty$--topology on $\fc(P)$).
In order to prove bijectivity of $\overline{\mathcal{H}}$, we will need to introduce the mixed bundles associated to representations $\rho\co\pi_1 M \to G$.  This will provide a proof of surjectivity.  Injectivity requires the idea that maps between bundles with the same holonomy can be described in terms of parallel transport.

Let $\rho\co\pi_1 M \to G$ be a representation.  We define the \emph{mixed bundle} $E_\rho$ by
$$E_\rho  = \wt{M}\cross_\rho G = \left(\wt{M}\cross G \right)\Big/(x,g)\sim (x\cdot \gamma, \rho(\gamma)^{-1}g)$$
Here $\wt{M}\stackrel{\pi_{\wt{M}}}{\maps}M$ is the universal cover of $M$, equipped with a basepoint $\wt{m}_0 \in \wt{M}$ lying over $m_0 \in M$.  
It is easy to check that $E_\rho$ is a principal $G$--bundle on $M$, with projection 
$[(\wt{m}, g)] \mapsto \pi_{\wt{M}}(\wt{m})$. We denote this map by $\pi_\rho\co E_{\rho} \to M$.  Note that since we have chosen basepoints $m_0\in M$ and $\wt{m}_0 \in \wt{M}$, $E_\rho$ acquires a canonical basepoint $[(\wt{m}_0, e)]\in E_\rho$ making $E_\rho$ a based principal $G$--bundle ($e \in G$ denotes the identity.)

The trivial bundle $\widetilde{M} \cross G$ has a natural horizontal connection, which descends to 
a canonical flat connection $\bbA_\rho$ on the bundle $E_{\rho}$.  This connection is given by 
$$\bbA_\rho\left([\tilde{x},g], \vect{v}_x\right)= Dq\left(\left(D_{\tilde{x}}\pi_{\wt{M}}\right)^{-1}(\vect{v}_x), \vect{0}_g\right).$$
On the left, $x \in M$, $\vect{v}_x \in T_xM$, $\tilde{x} \in \pi_{\wt{M}}^{-1}(x)\subset \wt{M}$, and $g \in G$.  On the right, $\vect{0}_g \in T_g G$ denotes the zero vector, $q$ denotes the quotient map $\wt{M}\cross G \to \wt{M} \cross_{\rho} G = E_\rho$, and $D_{\tilde{x}} \pi_{\wt{M}}$ is invertible because $\pi_{\wt{M}}\co \wt{M}\to M$ is a covering map.
We leave it to the reader to check that the connection $\bbA_\rho$ is flat, with holonomy representation 
$\mathcal{H}(\bbA_\rho) = \rho$.
This proves surjectivity of the holonomy map.  Injectivity will follow from:

\begin{proposition}\label{injectivity} Let $(P, p_0)$ and $(Q, q_0)$ be based principal $G$--bundles over $M$ with flat connections $A_P$ and $A_Q$, respectively.  If $\mathcal{H}(A_P) = \mathcal{H}(A_Q)$, then there is a based isomorphism $\phi\co P \to Q$ such that $\phi_\ast A_P = A_Q$.  
\end{proposition}

The proof will in fact show that the assumption of flatness in Proposition~\ref{injectivity} is unnecessary.  For non-flat connections, however, the holonomy along a loop $\gamma$ no longer depends only on the homotopy class of $\gamma$, so the condition $\mathcal{H} (A_P) = \mathcal{H} (A_Q)$ should be interpreted as saying that for every smooth loop $\gamma$ in $M$, the holonomies of $A_P$ and $A_Q$ around $\gamma$ coincide. 

The map $\phi$ is defined by setting $\phi(p_0 \cdot g) = q_0 \cdot g$ and then extending via parallel transport:
$$\phi(p) = T_\gamma^{A_Q} \circ \phi \circ T_{\overline{\gamma}}^{A_P},$$ where $\gamma\co [0,1] \to M$ is any path with $\gamma(0) = m_0$ and $\gamma(1) = \pi(p)$.  Using the fact that $\mathcal{H}(A_P) = \mathcal{H}(A_Q)$, one may check that $\phi$ is well-defined.  

To prove that $\phi_* A_P = A_Q$, we consider the lifts of a particular vector $\vect{v}\in T_m M$ under these two connections.  Let $\gamma \co [0,1]\to M$ be a smooth path with $\gamma(0) = m_0$ and $\gamma' (1/2) = \vect{v}$.  By definition of $\phi$ we have $\phi (\tilde{\gamma}_{p_0}) = \tilde{\gamma}_{q_0}$ and hence 
$D\phi \left(\tilde{\gamma}^\prime_{p_0} (t)\right) = \tilde{\gamma}^\prime_{q_0} (t)$ for any $t\in [0,1]$.  We now  have
\begin{eqnarray*}
(\phi_* A_P) (\tilde{\gamma}_Q (1/2), \vect{v}) 
& = & D\phi \left(A_P (\phi^{-1} \tilde{\gamma}_Q (1/2), \gamma' (1/2))\right) \hspace{.7in} \\
& = &D\phi \left(A_P (\tilde{\gamma}_P (1/2), \gamma' (1/2))\right) 
= D\phi \left(\tilde{\gamma}^\prime_P (1/2)\right)
 = \tilde{\gamma}^\prime_Q (1/2) \\
& = & A_Q (\tilde{\gamma}_Q (1/2), \gamma' (1/2))
 = A_Q (\tilde{\gamma}_Q (1/2), \vect{v})
\end{eqnarray*} 
and by $G$--equivariance it follows that $(\phi_* A_P) (q, \vect{v}) = A_Q (q, \vect{v})$ for every $q\in Q_m$.

This completes the proof of Proposition~\ref{hol}.
\end{proof}

As an easy consequence of this result, one obtains the more well-known bijection between (unbased) isomorphism classes of flat connections and conjugacy classes of homomorphisms.
A proof of the latter result is given by Morita~\cite[Theorem 2.9]{Morita-gcc}; however, Morita does not prove an analogue of Proposition~\ref{injectivity} and consequently his argument does not make the injectivity portions of these results clear.


\begin{thebibliography}{10}

\bibitem{Abraham-Smale}
Ralph Abraham.
\newblock {\em Lectures of {S}male on {D}ifferential {T}opology}.
\newblock Mimeographed notes. Columbia University, Dept. of Mathematics, New
  York, 1964 (?).

\bibitem{Abraham-Robbin}
Ralph Abraham and Joel Robbin.
\newblock {\em Transversal mappings and flows}.
\newblock An appendix by Al Kelley. W. A. Benjamin, Inc., New York-Amsterdam,
  1967.

\bibitem{Adem}
Alejandro Adem.
\newblock Characters and {$K$}-theory of discrete groups.
\newblock {\em Invent. Math.}, 114(3):489--514, 1993.

\bibitem{AMM}
Anton Alekseev, Anton Malkin, and Eckhard Meinrenken.
\newblock Lie group valued moment maps.
\newblock {\em J. Differential Geom.}, 48(3):445--495, 1998.

\bibitem{Atiyah-char-coh}
M.~F. Atiyah.
\newblock Characters and cohomology of finite groups.
\newblock {\em Inst. Hautes \'Etudes Sci. Publ. Math.}, (9):23--64, 1961.

\bibitem{A-B}
M.~F. Atiyah and R.~Bott.
\newblock The {Y}ang-{M}ills equations over {R}iemann surfaces.
\newblock {\em Philos. Trans. Roy. Soc. London Ser. A}, 308(1505):523--615,
  1983.

\bibitem{Atiyah-Segal}
M.~F. Atiyah and G.~B. Segal.
\newblock Equivariant {$K$}-theory and completion.
\newblock {\em J. Differential Geometry}, 3:1--18, 1969.

\bibitem{BDH}
G.~Baumslag, E.~Dyer, and A.~Heller.
\newblock The topology of discrete groups.
\newblock {\em J. Pure Appl. Algebra}, 16(1):1--47, 1980.

\bibitem{Carlsson-derived-rep}
Gunnar Carlsson.
\newblock Derived representation theory and the algebraic {$K$}-theory of
  fields.
\newblock arXiv:0810.4826, 2003.

\bibitem{Carlsson-derived}
Gunnar Carlsson.
\newblock Derived completions in stable homotopy theory.
\newblock {\em J. Pure Appl. Algebra}, 212(3):550--577, 2008.

\bibitem{Dask}
Georgios~D. Daskalopoulos.
\newblock The topology of the space of stable bundles on a compact {R}iemann
  surface.
\newblock {\em J. Differential Geom.}, 36(3):699--746, 1992.

\bibitem{Dask-Uhl}
Georgios~D. Daskalopoulos and Karen~K. Uhlenbeck.
\newblock An application of transversality to the topology of the moduli space
  of stable bundles.
\newblock {\em Topology}, 34(1):203--215, 1995.

\bibitem{DK}
S.~K. Donaldson and P.~B. Kronheimer.
\newblock {\em The geometry of four-manifolds}.
\newblock Oxford Mathematical Monographs. The Clarendon Press Oxford University
  Press, New York, 1990.
\newblock Oxford Science Publications.

\bibitem{EKMM}
A.~D. Elmendorf, I.~Kriz, M.~A. Mandell, and J.~P. May.
\newblock {\em Rings, modules, and algebras in stable homotopy theory},
  volume~47 of {\em Mathematical Surveys and Monographs}.
\newblock American Mathematical Society, Providence, RI, 1997.
\newblock With an appendix by M. Cole.

\bibitem{EM}
A.~D. Elmendorf and M.~A. Mandell.
\newblock Rings, modules, and algebras in infinite loop space theory.
\newblock {\em Adv. Math.}, 205(1):163--228, 2006.

\bibitem{FKK}
B.~Fine, P.~Kirk, and E.~Klassen.
\newblock A local analytic splitting of the holonomy map on flat connections.
\newblock {\em Math. Ann.}, 299(1):171--189, 1994.

\bibitem{Griffiths-Harris}
Phillip Griffiths and Joseph Harris.
\newblock {\em Principles of algebraic geometry}.
\newblock Wiley-Interscience [John Wiley \& Sons], New York, 1978.
\newblock Pure and Applied Mathematics.

\bibitem{Gruher}
Kate Gruher.
\newblock String {T}opology of {C}lassifying {S}paces.
\newblock Ph.D. thesis, Stanford University, 2007.

\bibitem{H-N}
G.~Harder and M.~S. Narasimhan.
\newblock On the cohomology groups of moduli spaces of vector bundles on
  curves.
\newblock {\em Math. Ann.}, 212, 1974/75.

\bibitem{Hatcher}
Allen Hatcher.
\newblock {\em Algebraic topology}.
\newblock Cambridge University Press, Cambridge, 2002.

\bibitem{Higman}
Graham Higman.
\newblock A finitely generated infinite simple group.
\newblock {\em J. London Math. Soc.}, 26:61--64, 1951.

\bibitem{Hironaka}
Heisuke Hironaka.
\newblock Triangulations of algebraic sets.
\newblock In {\em Algebraic geometry (Proc. Sympos. Pure Math., Vol. 29,
  Humboldt State Univ., Arcata, Calif., 1974)}, pages 165--185. Amer. Math.
  Soc., Providence, R.I., 1975.

\bibitem{Ho}
Nan-Kuo Ho.
\newblock The real locus of an involution map on the moduli space of flat
  connections on a {R}iemann surface.
\newblock {\em Int. Math. Res. Not.}, 2004.

\bibitem{Ho-Liu-ctd-comp-I}
Nan-Kuo Ho and Chiu-Chu~Melissa Liu.
\newblock Connected components of the space of surface group representations.
\newblock {\em Int. Math. Res. Not.}, (44):2359--2372, 2003.

\bibitem{Ho-Liu-moduli}
Nan-Kuo Ho and Chiu-Chu~Melissa Liu.
\newblock On the connectedness of moduli spaces of flat connections over
  compact surfaces.
\newblock {\em Canad. J. Math.}, 56(6):1228--1236, 2004.

\bibitem{Ho-Liu-ctd-comp-II}
Nan-Kuo Ho and Chiu-Chu~Melissa Liu.
\newblock Connected components of spaces of surface group representations.
  {II}.
\newblock {\em Int. Math. Res. Not.}, (16):959--979, 2005.

\bibitem{Ho-Liu-non-orient}
Nan-Kuo Ho and Chiu-Chu~Melissa Liu.
\newblock Yang--{M}ills connections on nonorientable surfaces.
\newblock {\em Comm. Anal. Geom.}, 16(3):617--679, 2008.

\bibitem{HSS}
Mark Hovey, Brooke Shipley, and Jeff Smith.
\newblock Symmetric spectra.
\newblock {\em J. Amer. Math. Soc.}, 13(1):149--208, 2000.

\bibitem{Kurzweil}
J.~Kurzweil.
\newblock On approximation in real {B}anach spaces.
\newblock {\em Studia Math.}, 14:214--231 (1955), 1954.

\bibitem{Lang-dg}
Serge Lang.
\newblock {\em Differential and {R}iemannian manifolds}, volume 160 of {\em
  Graduate Texts in Mathematics}.
\newblock Springer-Verlag, New York, third edition, 1995.

\bibitem{Lawson-thesis}
Tyler Lawson.
\newblock Derived {R}epresentation {T}heory of {N}ilpotent {G}roups.
\newblock Ph.D. thesis, Stanford University, 2004.

\bibitem{Lawson-prod}
Tyler Lawson.
\newblock The product formula in unitary deformation {$K$}-theory.
\newblock {\em $K$-Theory}, 37(4):395--422, 2006.

\bibitem{Lawson-simul}
Tyler Lawson.
\newblock The {B}ott cofiber sequence in deformation {$K$}-theory and
  simultaneous similarity in {$U(n)$}.
\newblock To appear in Math. Proc. Cambridge Philos. Soc., 2008.

\bibitem{Luck}
Wolfgang L{\"u}ck.
\newblock Rational computations of the topological {$K$}-theory of classifying
  spaces of discrete groups.
\newblock {\em J. Reine Angew. Math.}, 611:163--187, 2007.

\bibitem{Luck-Oliver}
Wolfgang L{\"u}ck and Bob Oliver.
\newblock The completion theorem in {$K$}-theory for proper actions of a
  discrete group.
\newblock {\em Topology}, 40(3):585--616, 2001.

\bibitem{McDuff-Segal}
D.~McDuff and G.~Segal.
\newblock Homology fibrations and the ``group-completion'' theorem.
\newblock {\em Invent. Math.}, 31(3):279--284, 1975/76.

\bibitem{Milnor-univ-bundles-2}
John Milnor.
\newblock Construction of universal bundles. {II}.
\newblock {\em Ann. of Math. (2)}, 63:430--436, 1956.

\bibitem{Mitter-Viallet}
P.~K. Mitter and C.-M. Viallet.
\newblock On the bundle of connections and the gauge orbit manifold in
  {Y}ang-{M}ills theory.
\newblock {\em Comm. Math. Phys.}, 79(4):457--472, 1981.

\bibitem{Morita-gcc}
Shigeyuki Morita.
\newblock {\em Geometry of characteristic classes}, volume 199 of {\em
  Translations of Mathematical Monographs}.
\newblock American Mathematical Society, Providence, RI, 2001.
\newblock Translated from the 1999 Japanese original, Iwanami Series in Modern
  Mathematics.

\bibitem{Morita-gdf}
Shigeyuki Morita.
\newblock {\em Geometry of differential forms}, volume 201 of {\em Translations
  of Mathematical Monographs}.
\newblock American Mathematical Society, Providence, RI, 2001.

\bibitem{Park-Suh}
Dae~Heui Park and Dong~Youp Suh.
\newblock Linear embeddings of semialgebraic {$G$}-spaces.
\newblock {\em Math. Z.}, 242(4):725--742, 2002.

\bibitem{Rade}
Johan R{\aa}de.
\newblock On the {Y}ang-{M}ills heat equation in two and three dimensions.
\newblock {\em J. Reine Angew. Math.}, 431:123--163, 1992.

\bibitem{Ramras-excision}
Daniel~A. Ramras.
\newblock Excision for deformation {$K$}-theory of free products.
\newblock {\em Algebr. Geom. Topol.}, 7:2239--2270, 2007.

\bibitem{Ramras-thesis}
Daniel~A. Ramras.
\newblock Stable representation theory of infinite discrete groups.
\newblock Ph.D. thesis, Stanford University, 2007.

\bibitem{Ramras-YM}
Daniel~A. Ramras.
\newblock The {Y}ang--{M}ills stratification for surfaces revisited.
\newblock Submitted. arXiv:0805.2587, 2008.

\bibitem{Segal-class-ss}
Graeme Segal.
\newblock Classifying spaces and spectral sequences.
\newblock {\em Inst. Hautes \'Etudes Sci. Publ. Math.}, (34):105--112, 1968.

\bibitem{Segal-cat-coh}
Graeme Segal.
\newblock Categories and cohomology theories.
\newblock {\em Topology}, 13:293--312, 1974.

\bibitem{Sepanski}
Mark~R. Sepanski.
\newblock {\em Compact {L}ie groups}, volume 235 of {\em Graduate Texts in
  Mathematics}.
\newblock Springer, New York, 2007.

\bibitem{Thaddeus}
Michael Thaddeus.
\newblock An introduction to the topology of the moduli space of stable bundles
  on a {R}iemann surface.
\newblock In {\em Geometry and physics (Aarhus, 1995)}, volume 184 of {\em
  Lecture Notes in Pure and Appl. Math.}, pages 71--99. Dekker, New York, 1997.

\bibitem{Uhl}
Karen~K. Uhlenbeck.
\newblock Connections with {$L\sp{p}$} bounds on curvature.
\newblock {\em Comm. Math. Phys.}, 83(1):31--42, 1982.

\bibitem{Wehrheim}
Katrin Wehrheim.
\newblock {\em Uhlenbeck compactness}.
\newblock EMS Series of Lectures in Mathematics. European Mathematical Society
  (EMS), Z\"urich, 2004.

\bibitem{Whitney}
Hassler Whitney.
\newblock Elementary structure of real algebraic varieties.
\newblock {\em Ann. of Math. (2)}, 66:545--556, 1957.

\bibitem{Wilansky}
Albert Wilansky.
\newblock {\em Topology for analysis}.
\newblock Robert E. Krieger Publishing Co. Inc., Melbourne, FL, 1983.
\newblock Reprint of the 1970 edition.

\end{thebibliography}
\end{document}